\documentclass[a4paper,10pt]{article}

\usepackage[utf8]{inputenc}
\usepackage[T1]{fontenc}
\usepackage{lmodern}
\usepackage[margin=2.5cm]{geometry}
\usepackage{amsmath,amssymb,amsfonts,amsthm,mathtools}
\usepackage{graphicx}
\usepackage{subfigure}
\usepackage{xcolor}
\usepackage{url}
\usepackage{microtype}
\usepackage[colorlinks=true,linkcolor=blue,citecolor=blue,urlcolor=blue]{hyperref}
\usepackage[capitalise,noabbrev]{cleveref}

\hypersetup{
  pdftitle={High-dimensional analysis of ridge regression for non-identically distributed data with a variance profile},
  pdfauthor={J. Bigot, I. Dabo, and C. Male}
}

\title{High-dimensional analysis of ridge regression for non-identically distributed data with a variance profile}
\author{J\'er\'emie Bigot \& Issa-Mbenard Dabo \& Camille Male\\
Institut de math\'ematiques de Bordeaux \& CNRS (UMR 5251)\\
Universit\'e de Bordeaux, France}
\date{January 23, 2025}

\numberwithin{equation}{section}

\theoremstyle{plain}
\newtheorem{theorem}{Theorem}[section]
\newtheorem{proposition}{Proposition}[section]
\newtheorem{lemma}{Lemma}[section]
\newtheorem{corollary}{Corollary}[section]
\theoremstyle{definition}
\newtheorem{definition}{Definition}[section]
\newtheorem{hypothesis}{Assumption}[section]

\crefname{hypothesis}{Assumption}{Assumptions}
\Crefname{hypothesis}{Assumption}{Assumptions}
\crefname{theorem}{Theorem}{Theorems}
\Crefname{theorem}{Theorem}{Theorems}
\crefname{proposition}{Proposition}{Propositions}
\Crefname{proposition}{Proposition}{Propositions}
\crefname{lemma}{Lemma}{Lemmas}
\Crefname{lemma}{Lemma}{Lemmas}
\crefname{corollary}{Corollary}{Corollaries}
\Crefname{corollary}{Corollary}{Corollaries}
\crefname{definition}{Definition}{Definitions}
\Crefname{definition}{Definition}{Definitions}

\DeclareMathOperator{\Tr}{Tr}
\DeclareMathOperator{\diag}{diag}
\DeclareMathOperator*{\argmin}{arg\,min}
\newcommand{\RR}{\mathbb{R}}
\newcommand{\CC}{\mathbb{C}}
\newcommand{\EE}{\mathbb{E}}
\newcommand{\one}{\mathbf{1}}
\newcommand{\JB}[1]{{\color{black}#1}}

\newenvironment{keywords}{\par\medskip\noindent\textbf{Keywords:}\ }{\par\medskip}

\begin{document}

\maketitle

\begin{abstract}
High-dimensional linear regression has been thoroughly  studied in the context of independent and identically distributed data. We propose to investigate high-dimensional regression models for independent  but non-identically  distributed data. To this end, we suppose that the set of observed predictors (or features) is a random matrix with a variance profile and with dimensions growing at a proportional rate. Assuming a random effect model, we study  the predictive risk of the ridge estimator for linear regression with such a variance profile. In this setting, we provide deterministic equivalents of this risk and of the degree of freedom of the ridge estimator. For certain class of variance profiles, our work  highlights the emergence of the well-known double descent phenomenon in high-dimensional  regression for the minimum norm least-squares estimator when the  ridge regularization parameter goes to zero. We also exhibit variance profiles for which the  shape of this predictive risk   differs from double descent. The proofs of our results are based on tools from random matrix theory in the presence of a variance profile that have not been considered so far to study regression models. Numerical experiments are provided to show the accuracy  of the aforementioned deterministic equivalents on the computation of  the predictive risk of ridge regression.  We also investigate the similarities and differences that exist with the standard setting of independent and identically distributed data.
\end{abstract}

% REQUIRED

\section{Introduction}

High-dimensionality is a subject of interest in the field of statistics, especially in regression problems, driven by the advent of massive data sets. This context gives rise to unexpected phenomena and contradictions with established statistical heuristics when the dimension $p$ of the predictors is fixed and the number $n$ of observations tends to infinity. These phenomena particularly appear  in the context of linear regression. Indeed, as the sample size and dimension of acquired data increase, the study of this model is different from the classical framework. In the asymptotic regime where $\min(n,p) \to + \infty$ and $\frac{p}{n} \to c > 0$, one can notably mention the occurrence of the double descent phenomenon corresponding to estimators that both interpolate the data and show good generalization performances \cite{Belkin19}.  This phenomenon contradicts the consensus heuristic that, when a model becomes over-parameterized, then the predictive risk increases due to overfitting of the training data and the model is no longer capable of generalizing. This double descent phenomenon has been thoroughly studied in the case of high-dimensional linear regression with independent and identically distributed (iid) data using tools from RMT, see e.g.\ \cite{21-AOS2133,bach2023highdimensional,belkin2020two} and references therein. For instance, when the entries of $X_n$ are independent centered random
variables with variance 1, the empirical prediction risk $\hat{r}_0^{\text{test}}(X_n)$ (which we will define later) admits the following expression \cite{21-AOS2133} when $n \to \infty$ and $p/n \to c$:
\begin{equation}
\lim_{n \to \infty, \; p/n \to c} \hat{r}_0^{\text{test}}(X_n)
=
\begin{cases}
\sigma^2 \dfrac{c}{1-c} + \sigma^2 \mbox{ if }  c < 1,\\
\alpha^2 \left(1 - \dfrac{1}{c} \right) + \sigma^2 \dfrac{1}{c-1} \mbox{ if }  c > 1,
\end{cases}
\quad \text{almost surely}.
\label{eq:quasibistochastic}
\end{equation}
This result shows that the prediction risk increases with $p$ as long as $p/n < 1$, reaches a peak around the interpolation threshold $p = n$, and then decreases when $p > n$. This non-monotonic behavior is classically referred to as the double descent phenomenon and is illustrated in \cref{fig:intro}(a).

 In this asymptotic setting, using tools from random matrix theory (RMT), many authors have therefore focused on the consequences of high-dimensionality on linear regression, see e.g.\ \cite{17-AOS1549,bach2023highdimensional,21-AOS2133,liao2018dynamics} and references therein. In this paper, we focus on the linear regression model
\begin{equation}\label{eq:linmod}
y = x^\top \beta_\ast + \varepsilon, 
\end{equation}
where $x\in\mathbb{R}^p$ is a vector of random predictors, $\varepsilon \in \RR$ is a noise vector independent of $x$ with $\mathbb E [\varepsilon] = 0$ and $\mathbb E [\varepsilon^2 ] = \sigma^2 >0$, $\beta_\ast \in \RR^p$ is a vector of unknown parameters, and $y \in \RR$ is the observed response. Let us consider a training sample $(x_1,y_1),...,(x_n,y_n)$ following this linear model, that is for $i = 1,...,n$, one has that $y_i = x_i^\top \beta_\ast + \varepsilon_i$ where the noise term $\varepsilon_i$ is defined as above. Then one obtains from \Cref{eq:linmod} that 
$
Y_n = X_n \beta_\ast + \varepsilon_n,
$
where $X_n = (x_1 | \dots | x_n)^\top \in \mathbb{R}^{n\times p}$, $Y_n = (y_1 | \dots | y_n)^\top \in \mathbb{R}^n$ and $\varepsilon_n = (\varepsilon_1 | \dots | \varepsilon_n)^\top \in \mathbb{R}^n$.

Classically, the predictors are assumed to be iid data, meaning that the rows of the matrix $X_n$ are independent vectors sampled from the same probability distribution. In this paper, we propose to depart from this assumption by considering the setting where the rows of  $X_n$ are  independent but non-identically distributed. To this end, we suppose that $X_n$ is expressed in the following form 
$
X_n = \Upsilon_n  \circ Z_n,
$
where $\circ$ denotes the Hadamard product between two matrices, $Z_n = (Z_{ij})$ has iid centered entries with variance one, and $\Upsilon_n = (\gamma_{ij}^{(n)})$ is a deterministic matrix. To simplify the notation, we shall sometimes write $\gamma_{ij}^{(n)} = \gamma_{ij}$ and thus drop the (possibly) dependence of $\gamma_{ij}$ on $n$. The matrix $\Gamma_{n} = (\gamma_{ij}^{2}) \in \RR^{n \times p}$ governs the variance of the entries of $X_n$, and it is called a {\it variance profile}. 

In the RMT literature, there  exist various works on the  analysis of the spectrum of large random matrices with a variance profile \cite{SHL96,HLJ07,ACDGM17,bigotmale,erdos2012,Ajanki2017}. In particular, we rely on results from \cite{HLJ07} to obtain a  deterministic equivalent of the spectral distribution of a data matrix $X_n = \Upsilon_n  \circ Z_n$ with a variance profile matrix $\Gamma_n$. The motivation for studying linear regression using such a variance profile is to consider the setting  where one has $n$ independent  pairs of observations $(Y_i,X_i)_{1 \leq i \leq n}$ (with $X_i = (X_{ij})_{1 \leq j \leq p}$) that are not necessarily identically distributed. Note that in the standard setting of iid data, one has that
$
\gamma_{ij}  = \gamma_{j} \quad \mbox{for all} \quad 1 \leq i \leq n, \mbox{ and } 1 \leq j \leq p.
$

The main goal of this paper is then to understand how assuming such a variance profile for  $X_n$ influences the statistical properties of ridge regression in the linear model \cref{eq:linmod} when compared to the standard assumption of iid observations. In this setting, our approach also allows to analyze the performances of the minimum norm least-squares estimator when the  ridge regularization parameter goes to zero.

The deterministic equivalents derived in this paper are not only of interest in the RMT context, but they also provide approximations of key quantities in high dimensions, namely the prediction error and the effective model complexity measured by its degrees of freedom. 
In the classical iid setting, such equivalents are typically used as a means to an end, enabling one to highlight regimes where ridge regression generalizes well or poorly,  characterize optimal regularization parameter, and  study ridge regression across the under and over-parameterized regimes. 
A main objective of the present work is to provide analogous consequences in the non-iid setting induced by a variance profile, and to identify which phenomena remain universal (e.g., the derivation of an optimal regularisation parameter in ridge regression) from those that depend on the variance profile  (e.g., the emergence of a multiple-descent phenomenon beyond the classical double descent one).

 As an application, our  methodology also offers a novel tool for analyzing data that arises from  a mixture model that is a classical framework in machine learning, when the data  are sampled from multiple underlying subpopulations or classes.  Indeed, consider latent class variables $C_1, \ldots, C_K$, which determine the class membership of each  feature vector $x_i$. Formally, for each $i$ and each class $k \in \{1, \ldots, K\}$, we assume that the latent class variable $C_i$ follows a categorical distribution:
$
\forall 1 \leq i \leq n, \quad \mathbb{P}(C_i = k) = \pi_k.
$
Within each class, the random predictor $x_i$  is then assumed to follow a specific covariance structure. Conditional on $\{C_i = k\}$, we model the predictors as
$
x_i = S_k^{1/2} x_i',
$
where $S_k = \text{diag}(s_{k,1}^2, \ldots, s_{k,p}^2)$ is a diagonal matrix that characterizes the covariance structure of the predictors within class $k$. The variances $s_{k,j}^2$ of the predictors vary across classes, reflecting potential heterogeneity in the data.  Then, given the class labels $C_1, \ldots, C_n$, the resulting matrix of predictors $X_n$ exhibits a variance profile governed by the matrix
$
\Gamma_{n} = (s_{C_i,j}^2) \in \mathbb{R}^{n \times p},
$
where each entry $s_{C_i,j}^2$ corresponds to the variance of the $j$-th feature for the $i$-th  feature vector, determined by its class membership $C_i$.  Hence, our approach allows to investigate the double descent phenomenon in the context of  high-dimensional ridge regression when the data follow a mixture model, an aspect that has not been previously explored in the literature,  to the best of our knowledge.

We consider the  high-dimensional context (with $p$ growing to infinity at a rate proportional to $n$) for which the   least squares estimator is possibly not uniquely defined.  Thus, we focus our analysis on the ridge regression estimator that is the minimizer of the following loss function
$
\hat{\theta}_{\lambda} = \argmin_{\theta \in \RR^p} \frac{1}{n} \| Y_n - X_n \theta\| ^2 + \lambda \| \theta \|^2,
$
for some regularization parameter $\lambda > 0$.  Regardless of the ratio between $n$ and $p$, this estimator has the following explicit expression
\begin{eqnarray}\label{eq:theta}
 \hat{\theta}_{\lambda} = (X^\top_n X_n + n \lambda I_p)^{-1} X_n^\top Y_n = X_n^\top(X_n X_n^\top + n\lambda I_n)^{-1} Y_n.   
\end{eqnarray}
Our analysis also includes the study of the minimum least-square estimator defined as
$$
\hat{\theta} = \argmin_{\theta \in \RR^p} \left\{ \| \theta\| \; : \; \theta \mbox { minimizes } \frac{1}{n} \|Y_n - X_n \theta\|^2 \right\},
$$
to which the  ridge regression estimator converges when $\lambda$ tends to zero. The estimator $\hat{\theta}$ is also known to be the solution found by gradient descent when initialized to zero, see e.g.\ \cite{21-AOS2133}[Proposition 1].

To study the statistical performances of the ridge regression estimator, we analyze its empirical predictive risk, denoted $\hat{r}_\lambda^{test} (X_n)$, and its train risk, denoted $\hat{r}_\lambda^{train} (X_n)$, defined as 
\begin{equation}
 \hat{r}_\lambda^{test} (X_n) = \EE[(\tilde{y} - \tilde{x}^\top \hat{\theta}_\lambda)^2 | X_n], \label{eq:defrisk}
\quad \mbox{ and } \quad
 \hat{r}_\lambda^{train} (X_n) = \frac{1}{n}\mathbb{E} [\parallel Y_n -  X_n \hat{\theta}_{\lambda} \parallel^2|X_n],
\end{equation}
 where  $(\tilde{y},\tilde{x}) \in \RR \times \RR^p $ is independent from $(Y_n, X_n)$ and satisfies
 $
\tilde{y}= \tilde{x}^\top \beta_\ast + \tilde{\varepsilon}, \mbox{ with } \EE[ \tilde{\varepsilon}] = 0, \;\EE[ \tilde{\varepsilon}^2] = \sigma^2 \ \mathrm{and} \ \widetilde x ,\  \tilde \varepsilon \  \mathrm{independent}.\label{eq:linmodtilde} 
 $
In the above formula, $\tilde{x} = \widetilde{S}_p^{1/2}  \widetilde{z}$ with $ \widetilde{z} \in \RR^p$  a random vector with iid centered entries and variance one, and $\widetilde{S}_p = \EE [\tilde{x} \tilde{x}^\top  ]= \diag (\tilde{\gamma}_1^2,\ldots,\tilde{\gamma}_p^2)$
denotes the variance profile of $\tilde{x}$. Note that the risk $\hat{r}_\lambda^{test} (X_n)$ is conditioned on the predictors $X_n$, and it is thus a random variable.

Following \cite{17-AOS1549}, we  focus on a random-effect hypothesis by assuming that the components of the vector $\beta_\ast$ are drawn independently at random. As argued in \cite{17-AOS1549}, this assumption corresponds to an average case analysis over a set of dense regression coefficients as opposed to the ``sparsity hypothesis'' \cite{booksparsity} or the ``manifold hypothesis'' \cite{liu2023linear} that are other popular assumptions in high-dimensional linear regression. The following assumptions are made throughout the paper, and they are used to derive deterministic equivalents of the training and predictive risks.
\begin{hypothesis} \label[hypothesis]{hyp:theta}
The vector $ \beta_\ast$ of regression coefficients is random, independent from $X_n$, $ \widetilde x$, $\varepsilon_n$ and $\tilde \varepsilon$, with $\EE[\beta_\ast] = 0$ and $\EE[\beta_\ast \beta_\ast^\top] = \frac{\alpha^2}{p} I_p.$
\end{hypothesis}

\begin{hypothesis}\label[hypothesis]{hyp:Z}
$\exists \ \delta > 0$ s.t.
$
\mathbb{E}[|Z_{ij}|^{4+\delta}],\mathbb{E}[| \widetilde{z}_{j}|^{4+\delta}] <+\infty,
$
$\forall \ 1 \leq i \leq n$ and $1 \leq j \leq p$.
\end{hypothesis}
\begin{hypothesis}\label[hypothesis]{hyp:var1}
$\exists \ \gamma_{\mathrm{max}} > 0$ s.t.
$
\underset{n\geq 1}{\mathrm{sup}}\underset{\substack{1\leq i \leq n \\ 1 \leq j \leq p}}{\mathrm{max}}\lbrace|\gamma_{ij}^{(n)}|,| \tilde\gamma_{j}^{(n)}|\rbrace <\gamma_{\mathrm{max}}.
$
\end{hypothesis}
 The coefficient $\alpha > 0$ represents the average amount of signal strength in model \cref{eq:linmod}. Suppose that \Cref{hyp:theta} holds true, then the expectation in \Cref{eq:defrisk} used to define the predictive risk is thus taken with respect to both the randomness of the vector of coefficients $\beta_\ast$, the vector $\tilde{x}$ and the additive noise $\varepsilon_n$ and $\tilde{\varepsilon}$. Assuming random regression coefficients  $\beta_\ast $ allows for a tractable theoretical characterization of the training and predictive risks and it captures essential aspects of high-dimensional behavior that are otherwise more difficult to analyze.  \cref{hyp:Z} with the supplementary \cref{hyp:var2} and \cref{hyp:dim_rec} (to be given later on) ensure that the spectrum of $\frac{1}{n}X_n^\top X_n$ behaves well in the high-dimensional regime, that is whenever $\min(n,p) \to + \infty$ and $\frac{p}{n} \to c > 0$. Indeed, \cite{HLJ07} proves that, under \cref{hyp:Z,hyp:var1}, the empirical singular value distribution of $X_n$ converges, in the high-dimensional regime, to a distribution that solves a fixed-point equation. Note that these assumptions are not limited to the case of random Gaussian data.

\subsection{Main contributions} \label{sec:maincontrib}

Recall that the estimation of $Y_n$ by ridge regression is 
$
\hat{Y}_{\lambda} = X_n \hat{\theta}_{\lambda} = A_{\lambda} Y_n, \quad \mbox{where} \quad A_{\lambda} = X_n (X^\top_n X_n + n \lambda I_p)^{-1} X_n^\top.
$
Then, the degrees of freedom (DOF) of the estimator $\hat{\theta}_{\lambda} $, that is defined as
$$
\widehat{df}_{1}(\lambda) =  \frac{1}{p}\Tr [A_{\lambda}] = \frac{1}{p}\Tr [\widehat{\Sigma}_n (\widehat{\Sigma}_n + \lambda I_p)^{-1}], \mbox{ where } \widehat{\Sigma}_n = \frac{1}{n}X_n^\top X_n,
$$
represents  the so-called effective dimension of the linear estimator $\hat{Y}_{\lambda}$. The DOF is widely used in statistics to define various criteria for model selection among a collection of estimators, see e.g.\ \cite{Efron04}. Inspired by recent results from \cite{bach2023highdimensional} in the setting of iid data, a first contribution of this work is to prove the following deterministic equivalence of the DOF 
\begin{equation}
\widehat{df}_{1}(\lambda) \sim  df_{1}(\lambda), \quad \mbox{where} \quad df_{1}(\lambda) =  \frac{1}{p}\Tr[ \Sigma_n ( \Sigma_n + \kappa(\lambda))^{-1}], \label{eq:equivDOF}
\end{equation}
where
$
\Sigma_n = \EE [ \widehat{\Sigma}_n ] =  \frac{1}{n}  \diag\left( \sum_{i=1}^{n}  \gamma^2_{i1} ,\ldots, \sum_{i=1}^{n}  \gamma^2_{ip}\right),
$
and $\kappa(\lambda)$ is diagonal matrix that depends upon the regularization parameter $\lambda$ and the variance profile matrix $ \Gamma_n$. Throughout the paper, the meaning of the equivalence notation $A_n \sim B_n$ between two random variables is
$
\lim_{n \to \infty, \; p/n \to c} |A_n - B_n| = 0, \; \mbox{almost surely}.
$
Hence, the equivalence relation \cref{eq:equivDOF} indicates that the DOF of the ridge regression estimator for the empirical covariance matrix $\widehat{\Sigma}_n$ corresponds to the DOF computed with its expected version $\Sigma_n$ (the usual population covariance matrix for iid data), and  another additive regularization  structure than $\lambda I_p$  that is  given by the diagonal matrix  $\kappa(\lambda)$ whose explicit expression is given in \Cref{sec:main}.

The notion of degrees of freedom (DOF) plays a central role in the selection of the regularization parameter $\lambda$. 
In principle, $\lambda$ can be optimally selected by minimizing the predictive risk, as done in Section \ref{sec:main}. 
However, this approach typically requires access to independent test data or assumes the use of underlying parameters such as the signal strength $\alpha$ which is typically unknown. To circumvent these limitations, a widely used data-driven alternative is the  Generalized Cross-Validation (GCV) criterion \cite{golub1979generalized}. 
The GCV provides an approximation of the out-of-sample prediction error using only training data, without requiring an explicit data-splitting procedure. 
 Following \cite{golub1979generalized}, we define in this paper the GCV criterion as
\begin{equation}
 \hat G(\lambda) = \frac{\hat{r}_\lambda^{\text{train}}(X_n)}{\left(1 - \widehat{df}_1(\lambda)\right)^2}. \label{eq:hatGCV}
\end{equation}
Minimizing the GCV thus amounts to selecting the value of $\lambda$ that achieves an optimal balance between goodness-of-fit, as captured by the training risk, and model complexity, as measured by the degrees of freedom. In Section \ref{sec:main}, we derive a deterministic equivalent of $ \hat G(\lambda)$ and we discuss its minimizers through numerical experiments.

Then, the second and main contribution of the paper is to derive deterministic equivalents for both the predictive risk $\hat{r}_\lambda^{\text{test}}(X_n)$ and the training risk $\hat{r}_\lambda^{\text{train}}(X_n)$ in the high-dimensional regime. These deterministic equivalents allow to quantify the influence of the aspect ratio $c = \lim_{n \to +\infty} p/n$ on generalization performance, as well as to analyze the impact of the signal strength $\alpha$. In particular, we show the existence of an optimal value $\lambda_\ast = \frac{\sigma^2p}{\alpha^2n}$ that minimizes both the empirical predictive risk and its deterministic equivalent for any variance profile (see \cref{cor:lambda_opt}), and that equals the one obtained for iid data in \cite{17-AOS1549}. This result reveals the robustness of optimal shrinkage with respect to heteroscedasticity in the predictors, and shows that the optimal scaling of ridge regularization is universal across a broad class of variance profiles.

The derivation of these deterministic equivalents relies on a precise control of resolvent quantities associated with the empirical covariance matrix $\widehat{\Sigma}_n$. 
More precisely, the expressions of both the training and predictive risks involve the diagonal entries of the resolvent $Q_p(z) = (\widehat{\Sigma}_n - z I_p)^{-1}$ and of its derivative. 
As a consequence, the analysis requires establishing a deterministic equivalent for the diagonal of the resolvent $\Delta[Q_p(z)]$, rather than its normalized trace.
We obtain such a result by extending and strengthening the resolvent analysis developed in \cite{HLJ07} for random matrices with a variance profile, thereby overcoming the lack of unitary invariance inherent to the non-identically distributed setting. In addition, we study the behavior of the predictive risk in the ridgeless limit $\lambda \to 0$ in order to analyze the statistical properties of the minimum-norm least-squares estimator. 
This regime is particularly delicate, as it is governed by the behavior of the spectrum of the population and sample covariance matrices near zero. 
A sharp local analysis of the spectrum around the origin is therefore required to control the limiting risk and to capture interpolation phenomena. Indeed, the behavior of the smallest non-zero eigenvalue of $\widehat\Sigma_n$ is closely linked to risk blow-ups in the double descent phenomenon.
To this end, we rely on the precise spectral results established in \cite{alt2017local}, which provide the necessary information on the low-end behavior of the spectrum and allow us to rigorously characterize the asymptotic predictive risk in the interpolating regime.

As a consequence of these results, we investigate the similarities and differences that exist between the standard setting of iid data and the one of non-identically distributed data with a variance profile. For example, if the variance profile is assumed to be quasi doubly stochastic in the sense that
\begin{equation}\label{eq:db_sto}
\frac{1}{n}\sum_{j=1}^{p} \gamma_{ij}^2 = \frac{p}{n}, \mbox{ for all } 1 \leq i \leq n,  \quad \mbox{and} \quad \frac{1}{n} \sum_{i=1}^{n} \gamma_{ij}^2 = 1, \mbox{ for all } 1 \leq j \leq p,
\end{equation}
then, we prove in this paper that $\hat{r}_0^{test}(X_n)$ satisfies Equation \cref{eq:quasibistochastic}
That is, the quasi doubly stochastic case corresponds to the known asymptotic limit of the predictive risk  of the minimum norm least squares estimator for iid data  \cite{21-AOS2133} when the entries of $X_n$ are independent centered random variables with variance $1$ which is referred to as a constant variance profile in this paper (that is $\gamma_{ij} = 1$). Consequently, the double descent phenomenon still arises for non-iid data as illustrated in \cref{fig:intro}(a).

As a third contribution by going beyond the quasi doubly stochastic assumption \cref{eq:quasibistochastic}, the results of this paper also allow to exhibit variance profiles for which the predictive risk has a shape that differs from double descent. This is illustrated in \cref{fig:intro}(b), where the deterministic equivalent of the predictive risk has a triple descent behavior (as a function of $p/n$) for the following piecewise constant variance profile
$$
\Gamma_n  = \left[\begin{array}{cc} \gamma_1^2  \one_{n/4}\one_{p/4}^\top  & \gamma_2^2   \one_{n/4}\one_{3p/4}^\top \\ \gamma_2^2   \one_{3n/4}\one_{p/4}^\top & \gamma_1^2   \one_{3n/4}\one_{3p/4}^\top  \end{array}\right], \label{eq:piecewise}
$$
where $\one_{q}$ denotes the vector of length $q$ with all entries equal to one, and $\gamma_1,\gamma_2$ are positive constant such that $\gamma_2 \gg \gamma_1$. Note that \cref{fig:intro}(b) also illustrates the accuracy of the deterministic equivalent of the predictive risk that is proposed in this paper.

\begin{figure}
\begin{center}
{\subfigure[]{\includegraphics[width = 0.49\textwidth]{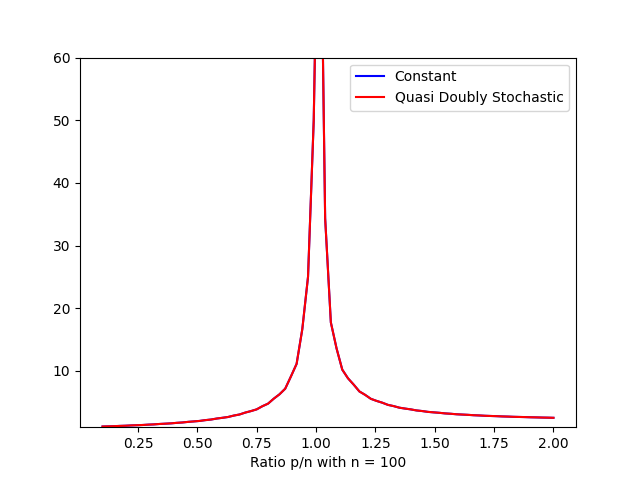}}
\label{fig:db_comparaison_const_db}}
\hfill
{\subfigure[]{\includegraphics[width =0.49\textwidth]{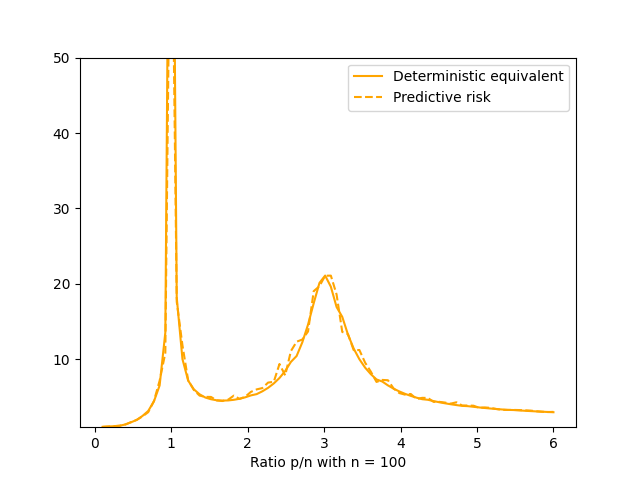}}
\label{fig:db_pw_Q_0.0005}}
\end{center}
\caption{Predictive risk for several variance profiles with $\lambda = 0$.  (a) Comparison of constant and quasi doubly stochastic variance profiles with $\alpha = \sigma = 1$, $n = 100$ and $p$ varying from $10$ to $200$. (b) Piecewise constant variance profile with $\gamma_1^2 = 0.0005$, $\gamma_2^2 = 1$, $\alpha = \sigma = 1$, $n = 100$ and $p$ varying from $10$ to $600$.} \label{fig:intro}
\end{figure}

\subsection{Organisation of the paper}

In \Cref{sec:related}, we review various works related to the analysis of high-dimensional linear regression and the study of random matrices with a variance profile. The main results are presented and discussed in \Cref{sec:main}. Numerical experiments are then reported in \Cref{sec:num}. A conclusion and some perspectives are proposed in \Cref{sec:conc}. All proofs using tools from the theory of random matrices and operator-valued Stieltjes transforms are deferred to an Appendix where we also discuss the use of random matrices with a variance profile in free probability.

\subsection{Publicly available source code}

For the sake of reproducible research,
a Python code is available at the following address:

\url{https://github.com/Issoudab/RidgeRegressionVarianceProfile} to implement  the experiments carried out in this paper.
\section{Related works} \label{sec:related}
In this section, we review the literature on the analysis of high-dimensional regression using tools from RMT. We also discuss existing works on linear regression for non-iid data, and the use of random matrices with a variance profile in statistics and RMT. 

\subsection{High-dimensional linear regression from the random matrix perspective}
When the sample size is comparable to the dimensionality of the observations,  recent advances in RMT have been successfully applied to various inference problems in high-dimensional  multivariate statistics, see e.g.\ \cite{reviewRMT} for a recent overview. Many works have considered the high-dimensional analysis of the linear model using tools from RMT for iid data with a general covariance structure $\Sigma \in \RR^{p \times p}$ (assumed to be a positive semi-definite matrix) that is for
$
X_n = Z_n \Sigma^{1/2}$,  for an $n \times p$ matrix $Z_n$ with iid centered entries having variance one.

In particular, for such data, the study of the minimum norm least-square estimator  and the double descent behavior of the predictive risk has been considered in \cite{21-AOS2133,bach2023highdimensional,belkin2020two,RichardsMR21}. The analysis of the predictive risk of ridge regression using iid data with a general covariance structure  has been studied in \cite{17-AOS1549,bach2023highdimensional}. More recently, the analysis of the minimum-norm least-squares estimator has enabled a substantial extension of asymptotic risk characterizations to  various learning paradigms. 
 In particular, \cite{song2024generalization,patil2024optimal} study the generalization properties of minimum-norm solutions in high-dimensional regimes, providing precise risk characterization {\JB in the settings of} out-of-distribution prediction and transfer learning,  building upon a rich line of earlier research that investigated ridge regression from a RMT perspective including \cite{elkaroui,dicker,li2023spectrum}. 
In parallel, closely related RMT techniques were developed in the wireless communications literature, where ridge-type estimators naturally arise in signal detection and channel estimation problems; see, for example, \cite{Couillet_Debbah_2011,TulinoV04}.

\subsection{Linear regression for independent but non-identically distributed data}

While the statistical analysis of linear regression for iid data with a general covariance structure is very well understood, the literature on the study of the linear model for non-identically distributed predictors appears to be  scarcer. A first analysis of  maximum likelihood estimation in standard models (including linear regression) for independent but non-identically distributed data dates back to  \cite{beran}. More recent works \cite{18-STS694,19-AOS1917}, on statistical inference in linear regression in the so-called model-free framework, allow to consider the setting on non-identically distributed predictors. The assumption of independent and identically distributed (iid) data often fails in real-world scenarios. This limitation notably appears in fields such as epidemiology \cite{eyre2022effect}, finance \cite{grobys2021we}, neuroscience \cite{tagliazucchi2013breakdown}, and climatology \cite{zscheischler2021evaluating}. The recognition of this issue has motivated significant research into the analysis of non-iid data. For example, \cite{zhang2022causal} highlights the non-iid nature of data in studies regarding the effectiveness of COVID-19 vaccines. They demonstrate that interactions between data points can introduce bias into estimation results. Beyond the risk of bias, the non-iid structure of data can also invalidate certain statistical methods. 
In response to the limitations of the iid setting, \cite{luo2024roti,atanasov2024risk} have recently proposed robust cross-validation schemes specifically tailored to settings with dependent observations, explicitly accounting for serial correlation and structural instability. Here, structural instability refers to features such as regime shifts, structural breaks, time-varying parameters, and changes in the underlying data-generating mechanism, all of which induce non-stationarities that violate the assumptions underlying standard validation procedures.
More broadly, the problem of cross-validation under dependence has attracted growing attention in the statistical and econometric literature. In particular, \cite{rabinowicz2022cross} provide a theoretical analysis of cross-validation procedures for dependent data, highlighting the failure modes of standard random-split approaches and establishing conditions under which modified validation schemes remain consistent, even in the presence of temporal dependence and mild non-stationarity. Similarly, \cite{yuval2025cross} study cross-validation in dependent settings, proposing dependence-aware resampling strategies that preserve the temporal structure of the data and yield more reliable risk estimates.
Together, these contributions underscore the necessity of adapting validation methodologies when data exhibit non-negligible dependence and structural instability, as is typically the case for financial time series such as realized volatility curves. 
The study of dependent data has also garnered growing attention in other domains. In control theory, researchers have made strides in addressing dependencies \cite{nagaraj2020least,tsiamis2023statistical,ziemann2024noise}. Similarly, the field of signal denoising has seen advancements with methods that explicitly account for data correlations \cite{sonthalia2023training,kausik2023double}, and the study of correlated data is a growing topic in high-dimensional statistics \cite{zhang2024spectral,zhang2024matrix}.

\subsection{The use of variance profile in RMT} \label{sec:varprofile}

RMT allows to describe the asymptotic distribution of the eigenvalues of large matrices with random entries, see e.g.\ \cite{MR2567175}. In particular, the well-known Marchenko-Pastur theorem characterizes the limiting spectral distribution of the covariance matrix $\widehat{\Sigma}_n = \frac{1}{n} X_n^\top X_n$ for a data matrix $X_n = Z_n \Sigma^{1/2}$ with iid rows in the asymptotic setting $\lim\limits_{n\to+\infty} \frac{p}{n} = c > 0$. A large class of random matrix models considered in RMT falls within the unitarily invariant framework \cite{meckes2014concentration}, in which the distribution of the random matrix is invariant under unitary (or orthogonal, in the real-valued case) conjugation. This structural property
has played a central role in the modern theoretical analysis of high-dimensional statistical estimators.
In particular, unitarily invariant random matrix models have been extensively studied in the context of linear ridge regression  \cite{li2023spectrum,17-AOS1549,bach2023highdimensional}.

However, in many applications (such as photon imaging \cite{SalmonHDW14}, network traffic analysis \cite{Bazerque2013}, ecology \cite{allesina2015predicting,allesina2015stability}, neuroscience \cite{aljadeff2015transition,aljadeff2015eigenvalues}, or genomics and microbiome studies \cite{Cao17}) one must go beyond this framework. Indeed, in these contexts, the amount of variance in the observed data matrix may vary substantially across samples, that is, across the rows of $X_n$. 
As a consequence, the assumption of homoscedasticity is no longer appropriate. The literature on statistical inference for high-dimensional random matrices with heteroscedastic structures has grown rapidly in recent years \cite{bigotmale,Bigot17,liu2018,MAL2016,Zhang18,Fan19}. One common approach to modeling heteroscedasticity is through the use of variance profiles. However, applying a variance profile to a matrix with i.i.d. entries breaks unitary invariance. This observation has led to the study of non–unitarily invariant models, often referred to as variance-profiled matrices.
In this setting, the dataset is modeled as a random matrix of the form
$
X_n = \Upsilon_n \circ Z_n,
$
with a variance profile to handle the setting of non-iid data has also found applications in the analysis of the performances of wireless digital communication channels  \cite{HLJ07}. Such matrices provide a flexible modeling framework for capturing heterogeneity in the data, in particular in the presence of missing or sparse observations by setting selected entries of the variance profile matrix to zero.

In the RMT literature, Hermitian random matrices with centered entries but non-equal distribution are referred to as generalized Wigner matrices for which many asymptotic properties are now well understood. For example, for Hermitian random matrices with a variance profile that is doubly stochastic (namely its rows and columns elements sump up to one), bulk universality at optimal spectral resolution for local spectral statistics have been established in \cite{erdos2012} and they are shown to converge to those of a standard Wigner matrix (that is with iid sub-diagonal entries).  The case of a generalized Wigner matrix with a variance profile that is not necessarily doubly stochastic has been studied in \cite{Ajanki2017}, and non-hermitian random matrices with a variance profile have been considered in \cite{cook2018,HLJ07,HACHEM2006649} using the notion of deterministic equivalent that consists in approximating the spectral distribution of a  random matrix by a deterministic function. A recent work \cite{bao2023leaveoneout} considers the analysis of approximate message passing for statistical estimation problems involving a data matrix with a variance profile, and an application to ridge regression to characterize the non-asymptotic distribution of the ridge estimator is proposed.

Let us now recall the key notion of Stieltjes transform.

\begin{definition}\label{def:Stieltjes}
Let $\mu$ be a probability measure supported on $\RR$. Then, its Stieltjes transform is defined as
$
g_{\mu}(z) = \int_{\RR} \frac{1}{t -z} d \mu(t), \quad \mbox{for} \quad z \in \CC^{+},
$
where $\CC^{+} = \{z \in \CC, \Im(z) > 0 \}$ and $\Im(\cdot)$ denotes the imaginary part of a complex number.
\end{definition}
Then, we build upon results from \cite{HLJ07} to construct deterministic equivalents  of the Stieltjes transforms (when $\lim_{n \to + \infty} p/n \to c >0$) of the empirical eigenvalue distribution $\hat{\mu}_n$ of $\widehat{\Sigma}_n$, and the empirical eigenvalue distribution $\tilde{\mu}_n$ of $\widetilde{\Sigma}_n = \frac{1}{n} X_n X_n^\top$ respectively.
$$
g_{\hat{\mu}_n}(z) = \frac{1}{p} \Tr[  (\widehat{\Sigma}_n -z I_p)^{-1}  ], \quad \mbox{and} \quad g_{\tilde{\mu}_n}(z) = \frac{1}{n} \Tr[  (\widetilde{\Sigma}_n -z I_n)^{-1}  ]
$$
for $z \in \CC \setminus \RR^+$.  Then, \cite{HLJ07}[Theorem 2.5] states that the following result holds.

\begin{theorem}\label{prop:Stieltjes} Suppose that \cref{hyp:Z} to \cref{hyp:dim_rec} hold true, then the following limit holds true almost surely
$$
\lim_{n \to \infty, \; p/n \to c}  \left( g_{\hat{\mu}_n}(z) - \frac{1}{p} \Tr [T_p(z)] \right) = 0, \quad \mbox{and} \quad \lim_{n \to \infty,\; p/n \to c}  \left( g_{\tilde{\mu}_n}(z) - \frac{1}{n} \Tr [\widetilde{T}_n(z)] \right) = 0, 
$$
for all $z \in \CC \setminus \RR^+$ with
$
T_p(z) = \diag(T_{p}^{(1)}(z),...,T_{p}^{(p)}(z))$ and $
\widetilde{T}_n(z) = \diag(\widetilde{T}_{n}^{(1)}(z),...,\widetilde{T}_{n}^{(n)}(z))
$
diagonal matrices of size $p \times p$ and $n \times n$ respectively, whose diagonal elements are the unique solutions of the deterministic system of $p + n$ equations
\begin{equation}\label{eq:T}
T_{p}^{(j)}(z)  =  \frac{-1}{z(1 + (1/n) \Tr[\widetilde{D}_j \widetilde{T}_n(z)])} \mbox{, } \widetilde{T}_{n}^{(i)}(z) =  \frac{-1}{z(1 + (1/n) \Tr[D_i T_p(z)])}, 
\end{equation}
with $1 \leq j \leq p$, $ 1 \leq i \leq n$,
$
\widetilde{D}_j = \diag( \gamma^2_{1j},\ldots,\gamma^2_{nj}) \mbox{ and }
D_i = \diag( \gamma^2_{i1},\ldots,\gamma^2_{ip}). 
$
Moreover, $ \frac{1}{p} \Tr [T_p(z)]$ and $\frac{1}{n} \Tr [\widetilde{T}_n(z)]$ are the Stieltjes transforms of probability measures denoted as $\nu_p$ and $\tilde{\nu}_n$ respectively.
\end{theorem}
The measures  $\nu_p$ and $\tilde{\nu}_n$, defined in \cref{prop:Stieltjes}, are called the deterministic equivalents of the  empirical eigenvalue distributions $\hat{\mu}_n$  and $\tilde{\mu}_n$ respectively. By a slight abuse of notation, we may also denote by $T_p(z)$  the vector of size $p$ whose entries are the coefficients $T_{p}^{(j)}(z)$ for $1 \leq j \leq p$ solutions of the fixed point Equations \cref{eq:T}. Currently, a classical method to numerically approximate the value of  $T_p(z)$  is to solve the nonlinear system of deterministic Equations \cref{eq:T} written in a vector form that is referred to as the Dyson equation in \cite{Ajanki2017,Ajanki2019,alt2017local,alt2018}.  Indeed, as stated in \cite{alt2017local}[Theorem 2.1], the vector $T_p(z)$ is known to be the unique solution of the Dyson equation
\begin{equation}
\frac{1}{T_p(z)} = -z I_p +  \frac{1}{n}\Gamma_n^\top \frac{1}{1+ \frac{1}{n}\Gamma_n T_p(z)}, \label{eq:Dyson}
\end{equation}
that corresponds to Equation \cref{eq:T} written in a vector  form, where $1/v$ has to be understood as taking the inverse of the elements of the vector $v$  entrywise. Similarly, the vector $\widetilde{T}_n(z) = (\widetilde{T}_n^{(i)}(z))_{1 \leq i \leq n}$ satisfies the Dyson equation
$
\frac{1}{\widetilde{T}_n(z)} = -z I_n +  \frac{1}{n}\Gamma_n \frac{1}{1+ \frac{1}{n}\Gamma_n^\top \widetilde{T}_n(z)}.
$
In this paper, we sometimes consider, as an illustrative example, the specific class of {\it quasi doubly stochastic} variance profiles defined in \Cref{eq:db_sto}. The fixed point equation \cref{eq:Dyson} typically does not have an explicit expression, and one has  to rely on numerical methods to solve it as done in \Cref{sec:num}. Nevertheless, when the variance profile is quasi doubly stochastic, the solution of the Dyson equation is a  vector $T_p(z)$ having constant entries equal to $m_p(z) \in \CC$ (or equivalently $T_p(z) = m_p(z) I_p $ is a scalar matrix) that satisfies
\begin{equation} \label{eq:m}
\frac{1}{m_p(z)} = -z +  \frac{1}{1 + \frac{p}{n} m_p(z)}, \quad \mbox{for all}  \quad z \in \CC \setminus \RR^+.
\end{equation}
The above equality corresponds to the well-known fixed point equation satisfied by the Stieltjes transform $m_p(z)$ of the  Marchenko-Pastur distribution \cite{MR2567175}. One also has that $\widetilde{T}_n(z) = \tilde{m}_n(z) I_n $ with $ \tilde{m}_n(z)$ satisfying
\begin{equation} \label{eq:tildem}
\frac{1}{\tilde{m}_n(z)} = -z +  \frac{p/n}{1 +   \tilde{m}_n(z)} , \quad \mbox{for all}  \quad z \in \CC \setminus \RR^+.
\end{equation}

\section{Main results} \label{sec:main}

In this section, we derive deterministic equivalents for the DOF, the GCV criterion and the predictive risk of ridge regression. We also obtain a deterministic equivalent of the predictive risk of minimum norm least square estimation when the ridge regularization parameter tends to zero. We compare these results to those that are already known in the standard setting of iid data, and we highlight the emergence of the double descent phenomenon for non-iid data. In the  literature dealing with iid data, these equivalents are typically used to identify regimes of good and bad prediction, and to design  tuning rules for $\lambda$.
In the variance-profile setting, a similar analysis can be carried out but the conclusions naturally split  into:
\begin{itemize}
\item[(a)]  Universal consequences that hold across all variance profiles, such as the optimal choice of $\lambda_\ast$ in \cref{cor:lambda_opt};
\item[(b)] Profile-dependent consequences, such as the shape of the ridgeless risk as a function of $c=p/n$, which can significantly deviate from the classical double descent behavior.
\end{itemize}
The remainder of \Cref{sec:main} establishes the deterministic equivalents needed to reach such conclusions.

\subsection{Degrees of freedom}\label{sec:dof}
In this section, we prove that the quantity
$
df_{1}(\lambda)
$
introduced in \Cref{sec:maincontrib} is a deterministic  equivalent of the DOF $\widehat{df}_1(\lambda)$ .

\begin{proposition} \label{prop:DOF}
Suppose that \cref{hyp:Z,hyp:var1} hold true, then for any $\lambda > 0$, one has that : 
\begin{eqnarray}
    \lim_{n \to \infty, \; p/n \to c} |\widehat{df}_1(\lambda) - df_1(\lambda)| = 0, \; \mbox{almost surely}, \label{eq:DOFprop}
\end{eqnarray}
with $df_1(\lambda) = \frac{1}{p}\Tr[ \Sigma_n ( \Sigma_n + \kappa(\lambda))^{-1}]$, and
\begin{eqnarray}
\kappa(\lambda) & = & \diag(\kappa_{1}(\lambda),\ldots, \kappa_{p}(\lambda) ),\mbox{ where } \kappa_{j}(\lambda) =  \frac{ \Tr[\widetilde{D}_j]}{  \Tr[\widetilde{D}_j \widetilde{T}_n(- \lambda)]}. \label{eq:kappa}   
\end{eqnarray}
\end{proposition}
When the variance profile is quasi doubly stochastic, one has that  $ \Tr[ \widetilde{D}_j ] = n$. Moreover, as remarked in \Cref{sec:varprofile},  the matrix $\widetilde{T}_n(- \lambda) = \tilde{m}_n(- \lambda) I_n $ is scalar, implying that for $1 \leq j \leq p$, $\Tr[\widetilde{D}_j \widetilde{T}_n(- \lambda)] = n  \tilde{m}_n(- \lambda)$. Moreover, in this setting, $\Sigma_n = I_p$. Consequently, if $\Gamma_n$ is quasi doubly stochastic, one has that
$
\kappa(\lambda) = \frac{1}{ \tilde{m}_n(- \lambda)} I_p,
$
and the deterministic equivalent of the DOF is
$
df_1(\lambda) = \frac{1}{p}\Tr\left[ \Sigma_n \left( \Sigma_n + \kappa(\lambda) \right)^{-1}\right] =   \left( 1+\frac{1}{ \tilde{m}_n(- \lambda)}\right)^{-1}.
$
The above equality corresponds to the formula of a deterministic equivalent of the DOF derived in \cite{bach2023highdimensional} for iid data when the entries of the features matrix $X_n$ are made of iid centered random variables with variances equal to $1$.

\subsection{Deterministic equivalents of the diagonal of the resolvent}

 Let $Q_p(z) = (\widehat \Sigma_n - zI_p)^{-1}$ for $z \in \CC \setminus \RR^+,
$
be the resolvent of the matrix  $\widehat{\Sigma}_n = \frac{1}{n}X_n^\top X_n$, and for any square matrix $A$,  we denote by $\Delta[A]$ the diagonal matrix whose diagonal entries are those of $A$. Then, as shown in the next subsection, a key argument to derive   deterministic equivalents of the predictive and training risks is to prove that the matrix $T_p(z)$,  defined by \Cref{eq:Dyson}, is a relevant deterministic equivalent of the diagonal matrix $\Delta[Q_p(z)]$ for an appropriate notion of asymptotic equivalence between matrices of growing size. This is the purpose of \cref{theoremRMT}  below that derives a stronger convergence result than the one stated in \cref{prop:Stieltjes} which only shows that $ \frac{1}{p} \Tr [T_p(z)]$ is a deterministic equivalent of $g_{\hat{\mu}_n}(z) = \frac{1}{p} \Tr [Q_p(z)]$.
 To this end, we define the following equivalence relation in order to specify an appropriate notion of deterministic equivalent for the matrix $\Delta[Q_p(-\lambda)]$.
\begin{definition}\label{def:equiv}
Let $\mathbf{A} = (A_p)_{p\geq 1}$ and $\mathbf{B} = (B_p)_{p\geq 1}$ be a family of  square complex random matrices such that for all $p\geq1$, $A_p, B_p \in \mathbb{C}^{p\times p}$. Then, the two families of matrices $\mathbf{A}$ and $\mathbf{B}$ are said to be equivalent, denoted by $\mathbf{A} \sim \mathbf{B}$, if 
\begin{eqnarray*}
\lim_{n \to \infty, \; p/n \to c} \left|\frac{1}{p} \mathrm{Tr}[A_p U_p] - \frac{1}{p} \mathrm{Tr}[B_p U_p]\right| = 0, \; \mbox{almost surely},
\end{eqnarray*}
for all family of deterministic matrices $(U_p)_{p\geq 1}$ satisfying:
\begin{eqnarray}\label{eq:diag_family}
\forall p\geq 1, \ \ \   U_p\in \mathbb{R}^{p\times p} & \mbox{ and } & \exists K >0, \ \sup_{p\geq 1}\max_{i\in [p]} |U_p^{(i)}| \leq K,
\end{eqnarray}
where $U_p^{(i)}$ denotes the i-th diagonal entry of $U_p$.
\end{definition}
Then using this definition, we extend \cref{prop:Stieltjes} from \cite{HLJ07} by proving that $T_p(z)$ is a relevant deterministic equivalent of $\Delta[Q_p(z)]$ in the following sense.
\begin{theorem}\label{theoremRMT}

For $z \in \mathbb{C} \setminus \mathbb{R}^+$, define the matrix families
\begin{align*}
\Delta[\mathbf{Q}(z)] &= \big( \Delta[Q_p(z)] \big)_{p \geq 1}, 
& \Delta[\widetilde{\mathbf{Q}}(z)] &= \big( \Delta[\widetilde Q_n(z)] \big)_{n \geq 1}, \\
\mathbf{T}(z) &= \big( T_p(z) \big)_{p \geq 1}, 
& \widetilde{\mathbf{T}}(z) &= \big( \widetilde T_n(z) \big)_{n \geq 1},
\end{align*}
and denote by $\mathbf{Q}'(z), \widetilde{\mathbf{Q}}'(z), \mathbf{T}'(z), \widetilde{\mathbf{T}}'(z)$ the corresponding families of derivatives with respect to $z$.

Assume that \cref{hyp:Z,hyp:var1} hold. Then, as $p,n \to \infty$, one has
\begin{align*}
\Delta[\mathbf{Q}(z)] &\sim \mathbf{T}(z), 
& \Delta[\mathbf{Q}'(z)] &\sim \mathbf{T}'(z), \\
\Delta[\widetilde{\mathbf{Q}}(z)] &\sim \widetilde{\mathbf{T}}(z), 
& \Delta[\widetilde{\mathbf{Q}}'(z)] &\sim \widetilde{\mathbf{T}}'(z).
\end{align*}

\end{theorem}
We postpone the proof of this theorem to \Cref{sec:proofs}  that is inspired from the proof of \cref{prop:Stieltjes} in \cite{HLJ07}.
This theorem is a cornerstone in the derivation of deterministic equivalents for $\hat{r}_\lambda^{test} (X_n)$ and $\hat{r}_\lambda^{train} (X_n)$  that leads to the main results of this paper provided in the next subsection.

\subsection{Deterministic equivalents of the  training  and predictive risks  for $\lambda > 0$} \label{sec:predrisk}

We first express  the training and predictive     risks in a more convenient way.

\begin{lemma} \label{lem:decomp}
For any $\lambda > 0$, the training risk $\hat{r}_\lambda^{train} (X_n) = \frac{1}{n}\mathbb{E} [\parallel Y_n -  X_n \hat{\theta}_{\lambda} \parallel^2|X_n]$ and the  predictive risk $\hat{r}_\lambda^{test} (X_n) = \EE[(\tilde{y} - \tilde{x}^\top \hat{\theta}_\lambda)^2 | X_n]$ have the following expressions
\begin{eqnarray}\label{eq:train_Chap1}
    \hat{r}_\lambda^{train} (X_n) & = &\frac{\lambda^2\alpha^2}{p}\mathrm{Tr}\Bigg[\widetilde Q_n(-\lambda) - \widetilde Q_n'(-\lambda)\Bigg] + \frac{\lambda^2\sigma^2}{n}\mathrm{Tr}[\widetilde Q_n'(-\lambda)],
\end{eqnarray}
\begin{eqnarray}\label{eq:risk}
    \hat{r}_\lambda^{test} (X_n) = \sigma^2 + \frac{\sigma^2}{n}\mathrm{Tr}\big[\widetilde{S}_p \Delta[Q_p(-\lambda)]\big] + \lambda\left( \frac{\lambda\alpha^2}{p} -  \frac{\sigma^2}{n} \right) \mathrm{Tr}\big[\widetilde{S}_p  \Delta[Q_p'(-\lambda)]\big],
\end{eqnarray}
where
$Q_p(z) = (\widehat \Sigma_n - zI_p)^{-1}$,respectively $\widetilde Q_n(z) = (\widetilde \Sigma_n - zI_p)^{-1}$, for $z \in \CC \setminus \RR^+,
$
is the resolvent of the matrix $\widehat{\Sigma}_n = \frac{1}{n}X_n^\top X_n$, respectively $\widetilde{\Sigma}_n = \frac{1}{n}X_n X_n^\top$ and $Q_p'(z)$, respectively $\widetilde Q_n'(z)$, denotes the derivative of $Q_p(z)$, respectively $\widetilde Q_n(z)$, with respect to $z$.
Moreover, we can exhibit the following Bias-Variance decomposition for $\hat{r}_\lambda^{test} (X_n)$: 
\begin{eqnarray*}
    \hat{r}_\lambda^{test} (X_n) & = & \sigma^2 + \mathrm{Bias}(\hat{\theta}_{\lambda}) + \mathrm{Var}(\hat{\theta}_{\lambda}),
\end{eqnarray*}
with 
\begin{eqnarray*}
\mathop{\rm Bias}\nolimits(\hat{\theta}_{\lambda}) = \EE\big[ (\mathbb{E}[\hat{\theta}_{\lambda}|X_n, \beta_\ast]  - \beta_\ast)^\top\widetilde{S}_p (\mathbb{E}[\hat{\theta}_{\lambda}|X_n, \beta_\ast] - \beta_\ast) \big|X_n\big]=\frac{\lambda^2\alpha^2}{p} \mathrm{Tr} \big[\widetilde{S}_p \Delta[Q_p'(-\lambda)]\big],
\end{eqnarray*}
\begin{eqnarray*}
\mathop{\rm Var}(\hat{\theta}_{\lambda})   =  \mathbb{E}\big[ (\hat{\theta}_{\lambda} - \mathbb{E}[\hat{\theta}_{\lambda}|X_n, \beta_\ast] )^\top\widetilde{S}_p (\hat{\theta}_{\lambda} - \mathbb{E}[\hat{\theta}_{\lambda}|X_n, \beta_\ast] ) \big|X_n\big]=\frac{\sigma^2}{n} \mathrm{Tr}\big[\widetilde{S}_p\Delta[Q_p(-\lambda)] -\lambda \widetilde{S}_p\Delta[Q_p'(-\lambda)] \big].
\end{eqnarray*}
\end{lemma}
We can now give a deterministic equivalent of the predictive risk $\hat{r}_\lambda^{test} (X_n)$ that is  obtained in a simple way by replacing the diagonal matrix $\Delta[Q_p(z)]$ with the deterministic matrix $T_p(z)$ in the expression \cref{eq:risk} of the predictive risk. All proofs of the following results are given in the Appendix.

\begin{theorem}\label{theoremStat}
Suppose that \cref{hyp:theta} to \cref{hyp:var1} hold true, then one can provide a deterministic equivalent for the predictive risk (respectively for the training risk), denoted by $r_\lambda^{test}$ (respectively $r_\lambda^{train}$) and defined as follows
\begin{eqnarray}\label{eq:equiv_det}
r_\lambda^{test} (X_n) &=& \sigma^2 + \frac{\sigma^2}{n}\mathrm{Tr}[\widetilde{S}_p T_p(-\lambda)] + \lambda\left( \frac{\lambda\alpha^2}{p} -  \frac{\sigma^2}{n} \right) \mathrm{Tr}[\widetilde{S}_p  T_p'(-\lambda)  ],\\
r_\lambda^{train} (X_n) &=& \frac{\lambda^2\alpha^2}{p}\mathrm{Tr}\Bigg[\widetilde T_n(-\lambda) - \lambda \widetilde T_n'(-\lambda)\Bigg] + \frac{\lambda^2\sigma^2}{n}\mathrm{Tr}[\widetilde T_n'(-\lambda)],
\end{eqnarray}
in the sense that it satisfies
$
\lim_{n \to \infty, \; p/n \to c} |\hat{r}_\lambda^{\bullet} (X_n) - r_\lambda^{\bullet} (X_n)| = 0, \; \mbox{almost surely},
$
where $\bullet \in \lbrace train, \ test \rbrace$ and $T_p'(z)$, respectively $\widetilde T_n'(z)$, denotes the  derivative of $T_p(z)$, respectively $\widetilde T_n(z)$,  with respect to $z$. 
\end{theorem}
From \cref{theoremStat}, one can deduce that, in the high-dimensional regime,  the training and predictive risk crystallize around deterministic values that  depend on the data only through their variance profile. These deterministic equivalents can be expressed explicitly whenever the variance profile is quasi doubly-stochastic since $T_p(z)$ is equal to the Marchenko-Pastur Stieltjes transform in this case (see \Cref{sec:varprofile}). However, the equivalents are not explicit in the general case since one does not  have an explicit formula for $T_p(z)$. Nevertheless, $T_p(z)$ being the solution of a fixed-point equation, it can be approximated through a fixed-point algorithm. This alternative allowed us to perform numerical experiments that testify the accuracy of  these deterministic equivalents (see \Cref{sec:num}).

We now study the choice of an optimal regularization parameter $\lambda$ using either the empirical predictive risk or its deterministic equivalent.

\begin{corollary}\label{cor:lambda_opt}
Suppose that \cref{hyp:theta} to \cref{hyp:var1} hold true. Then, both functions $\lambda \mapsto \hat r_\lambda^{test} (X_n)$ and $\lambda \mapsto r_\lambda^{test} (X_n)$ reach their minimum at $\lambda_{\ast} = \frac{\sigma^2p}{\alpha^2n}$ meaning that $\hat r_{\lambda_{\ast}}^{test}(X_n) \leq \hat r_\lambda^{test} (X_n)$ and
$r_{\lambda_{\ast}}^{test}(X_n) \leq r_\lambda^{test} (X_n)$, for any $\lambda > 0$.
\end{corollary}
Note that although the expression of $r_\lambda^{test} (X_n)$ depends on the variance profile, this is not the case for the optimal value of $\lambda_\ast$ that minimizes the predictive risk as shown by \cref{cor:lambda_opt}. Moreover, the optimal value $\lambda_\ast$ is also the one minimizing the predictive risk in the framework of \cite{17-AOS1549} with a matrix $X_n$ of features made of iid rows.  Hence, \cref{cor:lambda_opt} shows that an optimal regularization parameter derived in the standard iid setting remains valid under the more general non-iid framework using a variance profile, thereby reinforcing the universality of the optimal value $\lambda_{\ast}$ for predictive risk minimization.  Interestingly, both empirical and deterministic predictive risk achieve their minimum at the same optimal regularization parameter. Nevertheless, this strategy leads to the selection of $\lambda$ that requires to have access to test data or prior knowledge of the signal strength $\alpha$. Our GCV approach is closely related in spirit to \cite{luo2024roti}, despite the different assumptions considered in the two papers. Both works use high-dimensional random-matrix asymptotics to adapt GCV to ridge regression beyond the standard i.i.d. framework. While we study variance-profiled designs, \cite{luo2024roti} focuses on right-rotationally invariant data, allowing for dependence and heavy-tailed covariates. In both cases, the criterion is built from a spectral asymptotic analysis of the risk; the main difference is that, in their setting, standard GCV is biased, which leads to the corrected ROTI-GCV criterion.

Building upon the deterministic equivalents derived above and the use of the GCV criterion \eqref{eq:hatGCV}, one may also adopt the following alternative strategy for selecting the regularization parameter $\lambda$. Introducing the following deterministic equivalent
$$
 G(\lambda) = \frac{r_\lambda^{\text{train}}(X_n)}{\left(1 - df_1(\lambda)\right)^2},
$$
one obtains that
$$
\lim_{n \to \infty, \; p/n \to c} \big| \hat G(\lambda) - G(\lambda) \big| = 0
\qquad \text{almost surely}.
$$
Hence, in the proportional asymptotic regime, the empirical GCV criterion concentrates around a deterministic limit that depends on the variance profile only through the deterministic equivalents of the training risk and the degrees of freedom.
\begin{figure}[htbp]
\begin{center}
{\subfigure[]{\includegraphics[width = 0.49\textwidth]{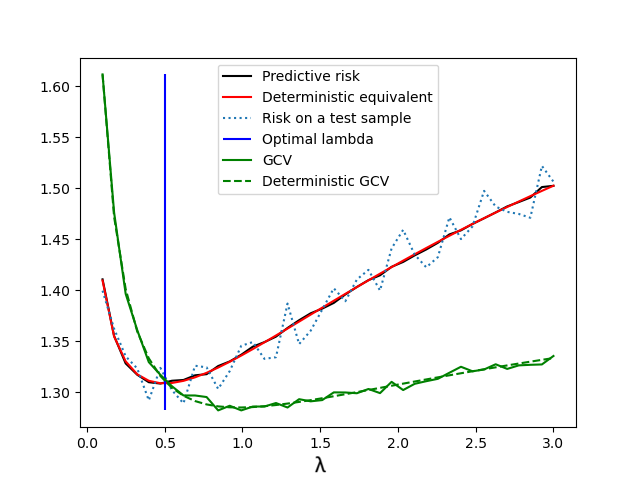}}}
\hfill
{\subfigure[]{\includegraphics[width = 0.49\textwidth]{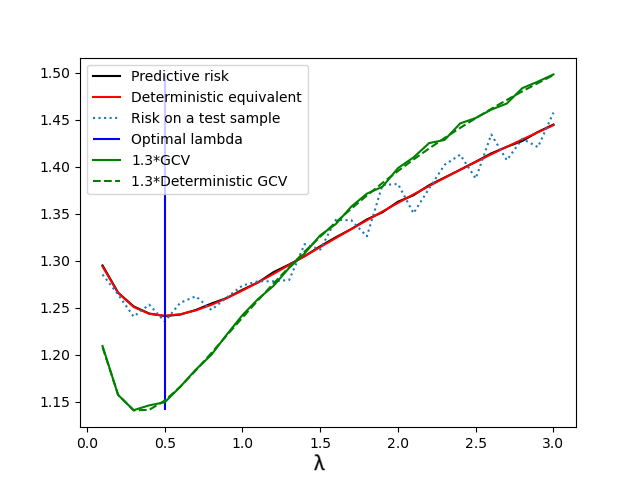}}}

{\subfigure[]{\includegraphics[width = 0.49\textwidth]{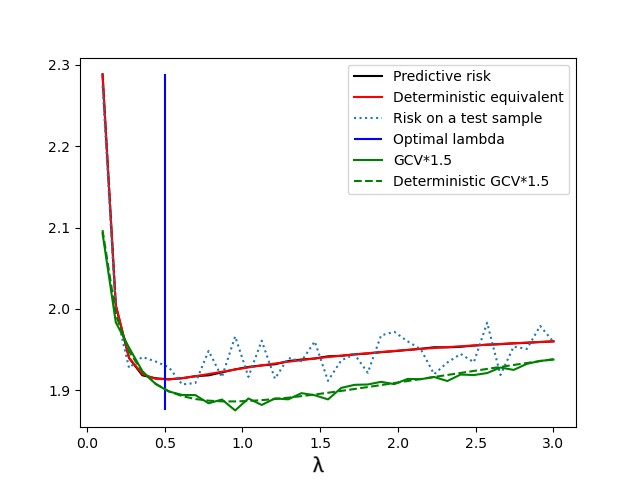}}}
\end{center}
\caption{ Comparison of the predictive risk, the GCV and their deterministic equivalents for several variances profiles: (a) block variance  profile, (b) alternated columns variance profile, (c) polynomial variance profile, as described in Section \ref{sec:num} on numerical experiments, with $\lambda $ ranging from $0.1$ to $3$, $\alpha = 1$, $\sigma = 1$, $n = 1000$ and $p = 500$.}  
\label{fig:GCV}
\end{figure}
 This naturally leads to the selection of a regularization parameter defined as
\begin{equation}
\lambda^{\mathrm{GCV}} \in \arg\min_{\lambda > 0} G(\lambda).
\end{equation}
 However, this estimator is not guaranteed to recover the oracle parameter $\lambda_\ast$ derived in \Cref{cor:lambda_opt} as illustrated in \Cref{fig:GCV}, where the predictive risk and the GCV criterion attain their minima at distinct, albeit nearby, values of $\lambda$. Hence, while the oracle parameter $\lambda_\ast$ is universal and does not depend on the variance profile, the GCV-selected parameter $\lambda^{\mathrm{GCV}}$ depends on it.

Since $G(\lambda)$ is defined in terms of the deterministic equivalent $\widetilde T(-\lambda)$, no explicit closed-form expression for $G(\lambda)$ is available in the general case. When the variance profile is quasi doubly stochastic, one has that $T_p(z) = m_p(z) I_p $ and
$
df_1(\lambda) =   \left( 1+\frac{1}{ \tilde{m}_n(- \lambda)}\right)^{-1},
$
where $m_p(\cdot)$ and $\tilde{m}_n(\cdot)$ are  Stieltjes transforms related to the  Marchenko-Pastur distribution satisfying \eqref{eq:m} and \eqref{eq:tildem}. Consequently, under the assumption of a  quasi doubly stochastic variance profile, we thus obtain that
$$
G(\lambda) =   \frac{1+\tilde{m}_n(- \lambda)}{\tilde{m}_n(- \lambda)}  \left( \frac{\lambda^2\alpha^2p}{n}\left(m_p(-\lambda) - \lambda m_p'(-\lambda)\right) + \frac{\lambda^2\sigma^2p}{n} m_p'(-\lambda) \right),
$$
but finding a closed-form minimizer of this functional remains unclear.

 \subsection{Deterministic equivalents of the  training  and predictive risks for $\lambda = 0$} \label{sec:predrisk2}
 
Note that the previous results hold for $\lambda > 0$, however one can extends the study for $\lambda = 0$, by considering the following additional hypotheses : 
\begin{hypothesis}\label[hypothesis]{hyp:var2}
$\exists \ \gamma_{\mathrm{min}} > 0$ s.t.
$
\forall n\geq 1,\underset{\substack{1\leq i \leq n \\ 1 \leq j \leq n}}{\mathrm{min}}\lbrace|\gamma_{ij}^{(n)}|,| \tilde\gamma_{j}^{(n)} |\rbrace \geq \gamma_{\mathrm{min}}.
$
\end{hypothesis}
\begin{hypothesis}\label[hypothesis]{hyp:dim_rec}
$\exists \ d_\ast > 0$ s.t. 
$
|\frac{p}{n} - 1| \geq d_\ast 
$
for any values of $n$ and $p$.
\end{hypothesis}
It is proved in \cite{alt2017local} that \cref{hyp:var2,hyp:dim_rec} ensure that the spectrum of $\frac{1}{n}X_n^\top X_n$ is bounded away from $0$, which is a necessary condition to study the $\lambda = 0$ case. Moreover, this case is of interest since it gives birth to the double descent phenomenon described in \Cref{sec:num}.

\begin{lemma}\label{lem:limzero}
Suppose that \cref{hyp:Z} to \cref{hyp:dim_rec} hold true, if $\frac{p}{n} < 1$ then $T_p(-\lambda)$ and its derivative, $T_p'(-\lambda)$ admit a limit when $\lambda$ tends to $0$. Indeed, there exists a constant $ \tau > 0$ such that 
$$
\lim_{\lambda \rightarrow 0, \; \lambda > 0} T_p (-\lambda) = \int_{ \tau}^{+\infty} \frac{\mu (dw)}{w} \quad \mathrm{and} \quad \lim_{\lambda \rightarrow 0, \; \lambda > 0} T_p'(-\lambda) = \int_{\tau}^{+\infty} \frac{\mu (dw)}{w^2},
$$
where $\mu = (\mu_{ij})$ is a positive $p\times p$ matrix valued-measure such that $\mu_{ii}$ is a probability measure and $\mu_{ij}$ is a null measure if $i\neq j$. We denote these limits by $T_p(0^-) = \lim\limits_{\lambda \rightarrow 0} T_p(-\lambda)$ and $T_p'(0^-) = \lim\limits_{\lambda \rightarrow 0} T_p'(-\lambda)$.

On the other hand, if $\frac{p}{n} > 1$ then $\kappa (\lambda)$ and its derivative $\kappa '(\lambda)$ also admit a limit when $\lambda$ tends to $0$ and we denote these limits by 
$$
 \kappa(0^+) = \lim_{\lambda \rightarrow 0, \; \lambda > 0}  \kappa(\lambda)  \quad \mathrm{and} \quad \kappa '(0^+) = \lim_{\lambda \rightarrow 0, \; \lambda > 0}\kappa '(\lambda),
$$
were $\kappa(\lambda)$ is the diagonal matrix defined in \Cref{eq:kappa}.
\end{lemma}
This lemma asserts that if $p<n$ (respectively $p>n$) then  $T_p(-\lambda)$ and $T'_p(-\lambda)$ (respectively $\kappa (\lambda)$) admit limits whenever $\lambda$ goes to $0$. These limits ensures that the deterministic equivalents provided by \cref{theoremStat} exist in the ridge (less) case, that is $\lambda = 0$ (See \cref{CorStat}).

\begin{corollary}\label{CorSto}
Suppose that \cref{hyp:Z} to \cref{hyp:dim_rec} hold true and that the variance profile $\Gamma_n$ is quasi doubly stochastic. Then, 
\begin{itemize}
\item[-] if $\frac{p}{n} < 1$, one has that
$
T_p(0^-) = m_p(0)I_p,
$
\item[-] if $\frac{p}{n} > 1$, 
$
\kappa (0^+) = \frac{1}{\tilde{m}_n(0)}I_p\quad \mathrm{and} \quad \kappa '(0^+) = \frac{\tilde{m}_n'(0)}{\tilde{m}_n^2(0)}I_p,
$
\end{itemize} 
where $m_p(.)$, resp.\ $\tilde{m}_n(.)$, is the Stieltjes transform of the Marchenko-Pastur distribution with parameter $p/n$, resp.\ $n/p$.
\end{corollary}
The deterministic equivalents provided in \cref{theoremStat} heavily depend on $T_p(-\lambda)$ and $\kappa(\lambda)$ and \cref{CorSto} ensures that these values can be evaluated for $\lambda = 0$ as long as $p\neq 0$. Hence, one can understand the behavior of $r_\lambda^{test} (X_n)$ when $\lambda \rightarrow 0$, that depends on the ratio $\frac{p}{n}$. The following corollary exploits \cref{CorSto} in order to describe this behavior that gives birth to the double descent phenomenon as we will see in \Cref{sec:num}.

\begin{corollary}\label{CorStat}
Suppose that \cref{hyp:theta} to \cref{hyp:var1} hold true, then the limit of the deterministic equivalent $r_\lambda^{test} (X_n)$ of the predictive risk as $\lambda \rightarrow +\infty$ is 
$$
\lim_{\lambda \rightarrow +\infty} r_\lambda^{test}(X_n) = \frac{\alpha^2}{p}\mathrm{Tr}[\widetilde{S}_p ] + \sigma^2,
$$
Moreover, supposing that \cref{hyp:theta} to \cref{hyp:dim_rec} hold true, then  the limit of the deterministic equivalent  $r_\lambda^{test} (X_n)$ , as $\lambda \rightarrow 0$,  is as follows
\begin{itemize}
\item[-] if $\frac{p}{n} < 1$ then 
$
\lim_{\lambda \rightarrow 0^+} r_\lambda^{test}(X_n) =  \frac{\sigma^2}{n}\mathrm{Tr}[\widetilde{S}_p  T_p(0^-)] + \sigma^2,
$
\item[-] if $\frac{p}{n} > 1$ then 
$
\lim_{\lambda \rightarrow 0^+} r_\lambda^{test}= \frac{\alpha^2}{p}\mathrm{Tr}[\widetilde{S}_p \kappa (0^+) (\Sigma_n + \kappa (0^+))^{-1}] + \frac{\sigma^2}{n}\mathrm{Tr}[\kappa '(0^+)\widetilde{S}_p \Sigma_n (\Sigma_n + \kappa (0^+))^{-2}] + \sigma^2,
$

where  $\kappa'(\lambda)$ denotes the derivative with respect to $\lambda$ of the diagonal matrix $\kappa(\lambda)$ defined in \Cref{eq:kappa}.
\end{itemize}
\end{corollary}
In the case $\lambda \rightarrow 0$, \cref{CorStat} yields a deterministic equivalent of the predictive risk of the minimum norm least-square estimator $\hat{\theta}$. Then, it can be seen that the risk  $r_0^{test}(X_n)$ exhibit two different  behaviors depending on the value of the ratio $p/n$ with respect to one. Moreover, it can be seen in \Cref{sec:num} that $r_0^{test}(X_n)$ seems to be an accurate estimation of the actual predictive risk in the ridgeless case, namely $\hat{r}_0^{test}(X_n)$.

If $\frac{p}{n} < 1$, then  $\hat{\theta}_0$ is equivalent to the ordinary least-square estimator $\hat{\theta} = (X_n^\top  X_n)^{-1}X_n^\top Y_n$ which is known to be unbiased, and the risk  $r_0(X_n)$ is thus only composed of a variance term. Moreover, if the variance profile is assumed to be quasi-bistochastic, then, given that
$$
T_p(0^-) = m_p(0)I_p \quad \mbox{and} \quad m_p(0) = \frac{1}{1-p/n},
$$
it follows that
$
r_0(X_n) =  \frac{\sigma^2}{1-\frac{p}{n}} \times \frac{1}{n}\mathrm{Tr}[\widetilde{S}_p] + \sigma^2 = \sigma^2 \frac{\frac{p}{n}}{1-\frac{p}{n}}  + \sigma^2 \rightarrow \sigma^2 \frac{c}{1-c}  + \sigma^2, 
$
as $n \to \infty, \; \frac{p}{n} \to c$, which corresponds, when $c < 1$, to the known asymptotic limit of the predictive risk of least squares estimation  for iid data 
 when the entries of the features matrix $X_n$ are iid centered random variables with variances equal to $1$, see e.g.\ \cite{21-AOS2133}[Proposition 2].
 
If $p/n > 1$, then the deterministic equivalent of the predictive risk is composed of a bias term and a variance term. If the variance profile is assumed to be quasi-bistochastic, the values of these two terms can be made more explicity as follows. In this setting, given that $\Sigma_n = I_p$, $\tilde{m}_p(0) = \frac{1}{p/n-1}$ and $\tilde{m}'_p(0) = \frac{p/n}{(p/n-1)^3}$, one has that
$$
\kappa (0^+) = (p/n-1)I_p\quad \mathrm{and} \quad \kappa '(0^+) = \frac{p/n}{p/n-1}I_p.
$$
Hence, using that $\frac{1}{n}\mathrm{Tr}[\widetilde{S}_p] = p/n$, we finally obtain that
$
r_0(X_n) = \alpha^2 \left(1 -\frac{n}{p} \right) + \sigma^2 \frac{1}{p/n-1}  \rightarrow \alpha^2 \left(1 -\frac{1}{c} \right) + \sigma^2 \frac{1}{c-1},
$
as $n \to \infty, \; \frac{p}{n} \to c$, which corresponds, when $c > 1$, to known results on the bias-variance decomposition of the asymptotic limit of the predictive risk  of the minimum norm least squares estimator  for iid data 
 when the entries of $X_n$ are iid centered random variables with variance $1$, see e.g.\ \cite{21-AOS2133}[Theorem 1].

Beyond the assumption of a quasi-stochastic variance profile, it is difficult to analytically determine the shape of the predictive risk $r_0(X_n)$ as a function of the ratio $p/n$. Indeed, this requires to at least know upper and lower bounds on the magnitude of the diagonal elements of $T_p(0^-)$ and $\kappa(0^+)$ (and its derivative). As  shown by \cref{lem:limzero} when $p/n< 1$, this issue  amounts to finding upper and lower bounds of the support of the matrix-valued measure $\mu$  satisfying
$
T_p (-0) = \int_{ \tau}^{+\infty} \frac{\mu (dw)}{w}.
$
Hence, upper and lower bounding $T_p (-0)$ is   related to understanding the value of the constant $\tau$ and  the size of the support of the limiting spectral distribution of the covariance matrix $\Sigma_n$  which remains (to the best of our knowledge) an open problem for random matrices with  an arbitrary variance profile.

 Nevertheless, in \Cref{sec:num}, we use \cref{CorStat} and computational methods to evaluate $T_p(0^-)$ and $\kappa(0^+)$ to report numerical experiments  illustrating that the double descent phenomenon for the predictive risk of $\hat{\theta}$ also holds in the high-dimensional model \cref{eq:linmod} with more general variance profile than a quasi-bistochastic one. In \Cref{sec:num}, we also exhibit variance profiles for which the predictive risk has a shape that differs from double descent.

\section{Numerical experiments}  \label{sec:num}
In this section, we illustrate the results of this paper with  numerical experiments. Numerical experiments  have been conducted  using a random matrix $Z_n = (z_{ij})_{\substack{1 \le i \le n \\ 1 \le j \le p}}$ with iid entries sampled from a centered and normalized Pareto distribution. More precisely, let $(\xi_{ij})_{\substack{1 \le i \le n \\ 1 \le j \le p}}$ be iid random variables following a Pareto distribution with parameters $(a,m) = (6,1)$, whose density is given by
$
f(x) = \frac{a m^a}{x^{a+1}} = \frac{6}{x^7}, \ x \ge 1.
$
The entries of $Z_n$ are then defined as
$
z_{ij} = \frac{\sqrt{a}\,(\xi_{ij} - a)}{\sqrt{a - 2}},
$
so that $\mathbb{E}[z_{ij}] = 0$ and $\mathbb{V}ar(z_{ij}) = 1$.
The moments of the Pareto random variables $\xi_{ij}$ are given by
$
\mathbb{E}[\xi_{ij}^k] =
\begin{cases}
\dfrac{a m^k}{a - k} = \dfrac{6}{6 - k}, & \text{if } k < a, \\
+\infty, & \text{otherwise}.
\end{cases}
$
In particular, the first five moments of $\xi_{ij}$ are finite. Consequently, the entries of $Z_n$ satisfy ~\cref{hyp:Z}.  
The test data are generated according to the same procedure. We note that we also conducted the same experiments using Gaussian random variables and obtained quantitatively identical numerical results. This observation provides additional empirical evidence in support of the universality of our theoretical findings with respect to the distribution of the design entries.

In addition, we consider several variance profiles  normalized such that 
$
\frac{1}{n}\sum_{i = 1}^n\sum_{j=1}^p \frac{\gamma^2_{ij}}{p} = 1.
$
Apart from the constant, quasi doubly stochastic, and piecewise constant variance profiles mentioned earlier, we will use the following examples of variance profiles:
\begin{description}
\item[-] The alternated columns variance profile satisfying $ \gamma_{ij} = \gamma_1 $ if $ j $ is even and $ \gamma_{ij} = \gamma_2 $ if $ j $ is odd.
\item[-] The polynomial variance profile satisfying $ \gamma_{ij}^2 = \Big|\frac{i-j}{\mbox{min}(n,p)}\Big|^6 + \tau $ for some $\tau > 0$.
\item[-] The block variance profile
$$
\Gamma_n  = \left[\begin{array}{ccc} \gamma_1^2  \one_{n/4}\one_{p/4}^\top  & \gamma_2^2   \one_{n/4}\one_{p/3}^\top  & \one_{n/4}\one_{5p/12}^\top  \\ \gamma_2^2   \one_{n/3}\one_{p/4}^\top  & \gamma_1^2\one_{n/3}\one_{p/3}^\top & \gamma_3^2   \one_{n/3}\one_{5p/12}^\top   \\    \one_{5n/12}\one_{p/4}^\top  & \gamma_3^2\one_{5n/12}\one_{p/3}^\top & \gamma_1^2   \one_{5n/12}\one_{5p/12}^\top  \end{array}\right], 
$$
for some sufficiently different constants $\gamma_1,\gamma_2,\gamma_3$.
\item[-]  A variance profile referred to as the Berlin Photo, displayed in \cref{fig:berlin_photo}, for which $\gamma_{ij}$ is equal to the value of the coordinate $(i,j)$ of the pixel of the image shown in \cref{fig:berlin_photo}.
\end{description}

\begin{figure}[htbp]
\begin{center}
\includegraphics[height=8.5cm,width=10cm]{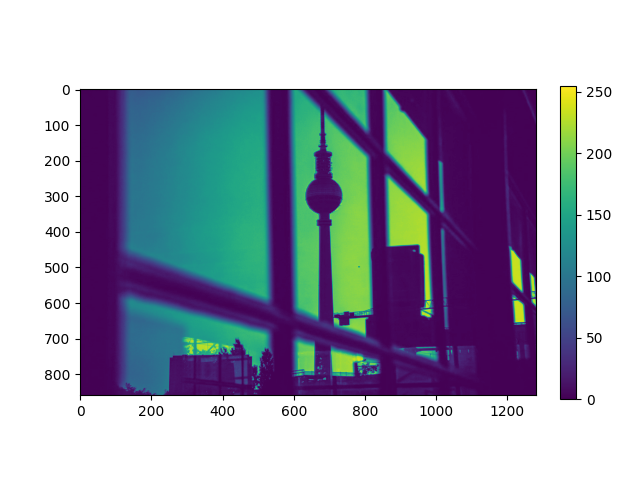}
\caption{The Berlin Photo variance profile whose entries correspond to the green channel of the pixels of a RGB picture taken in Berlin.}
\label{fig:berlin_photo}
\end{center}
\end{figure}

The values of the predictive and training risks and their deterministic equivalent are compared for various variance profiles in \cref{fig:profile_comparaisons,fig:profile_train_comparaisons} with $ \lambda $ ranging from 0.1 to 10, $n = 400$ and $p = 600$. The curves displayed in these figures  confirm that $r_\lambda^{test}(X_n)$ and $r_\lambda^{train}(X_n)$ are very accurate estimators of $\hat r_\lambda^{test}(X_n)$ and $\hat r_\lambda^{train}(X_n)$ in high-dimension since the dashed and solid lines coincide. These variance profiles also provide curves of predictive risks having similar shapes (up to a vertical translation). The curves representing the cases of constant and quasi doubly stochastic variance profile coincide, confirming the comment made in \Cref{sec:maincontrib} on the doubly stochastic profile and its similarities with the setting of iid data. Moreover, the minimum of $ r_\lambda^{test} (X_n) $ is indeed reached at the optimal value $ \lambda_\ast $ for any variance profile as shown by \cref{cor:lambda_opt}.

\begin{figure}[htbp]
\begin{center}
{\subfigure[]{\includegraphics[width = 0.49\textwidth]{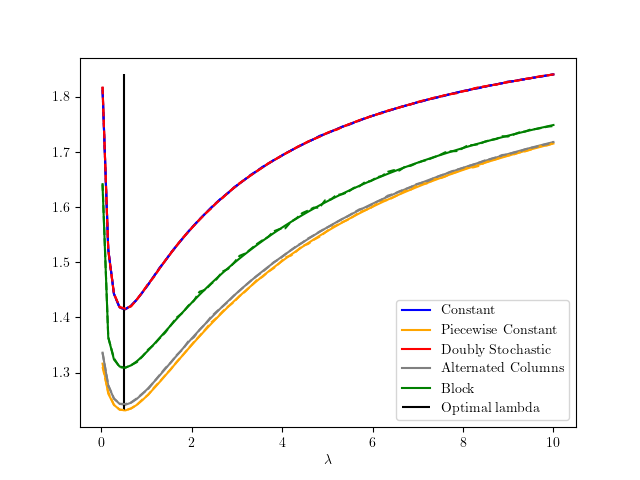}}
\label{fig:profile_comparaisons}}
\hfill
{\subfigure[]{\includegraphics[width =0.49\textwidth]{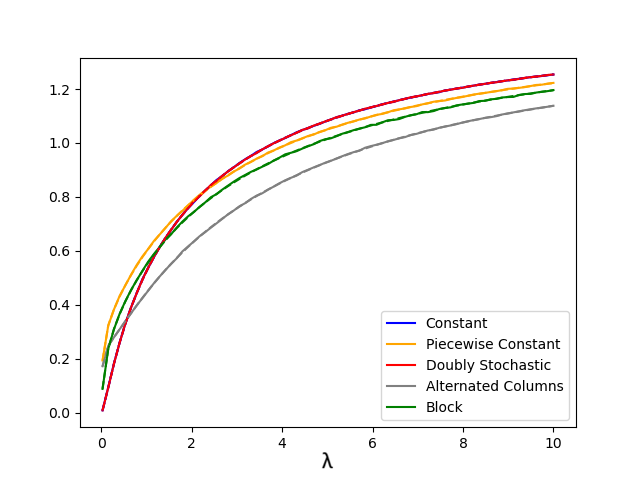}}
\label{fig:profile_train_comparaisons}}
\end{center}
\caption{Training and predictive risk for several variance profiles with $\lambda $ ranging from $0.1$ to $10$, $\alpha = 1$, $\sigma = 1$, $n = 400$ and $p = 600$.  (a) Comparison of $\hat r_\lambda^{test}(X_n)$  and $r_\lambda^{test}(X_n)$   for several variance profiles. (b) Comparison of $\hat r_\lambda^{train}(X_n)$  and $r_\lambda^{train}(X_n)$   for several variance profiles. The dashed lines correspond to the risks while the solid lines correspond to the deterministic equivalents.} \label{fig:comparaison}
\end{figure}

We illustrate the appearance of a double or triple descent phenomenon for several variance profiles in \cref{fig:descent_vp-min} (a), (c) and (e). These figures represent the predictive risk and its approximation by a deterministic equivalent for various values of the ratio $p/n$ with $\lambda = 0$. The solid lines depict $ r_0^{test}$ whereas the dashed lines represent $ \hat r_0^{test}$. For every variance profile, the solid and dashed lines coincide which confirms that $ r_0^{test}$ is a relevant deterministic equivalent of $ \hat r_0^{test}$. For the constant and quasi doubly stochastic variance profiles, we observe the well known double descent phenomenon as illustrated by \cref{fig:descent_vp-min} (a) since the curves are increasing for $p/n < 1$ and decreasing for $p/n > 1$. This is related to \cref{CorStat} which states that the expression of $\lim_{\lambda \rightarrow 0} r_\lambda^{test} (X_n)$ depends upon the value of the ratio $p/n$ with respect to one. Nevertheless, for other variance profiles, the shape of the predictive risk can be very different from one variance profile to another, and it differs from the usual double descent. For some variance profiles a phenomenon of triple descent arises, notably for the piecewise constant and block cases, as shown in \cref{fig:descent_vp-min} (c) and (e). We also remark the appearance of a quadruple descent in the case of the polynomial profile, as shown in \cref{fig:descent_vp-min} (e).

As already remarked at the end of \Cref{sec:main}, it is difficult to analytically explain this phenomenon for an arbitrary variance profile since we do not have an explicit formula for $T_p(z)$ in the general case. Nevertheless, as nicely discussed in \cite{schaeffer2023double}, the double descent phenomenon is very much related to the distance from zero of the smallest non-negative eigenvalue of  $\widehat{\Sigma}_n$. Having this eigenvalue close to zero causes double descent. In \cref{fig:descent_vp-min} (b), (d) and (f), we thus represent the smallest non-zero eigenvalue $\tau_{\min}(\widehat{\Sigma}_n)$ of $\widehat{\Sigma}_n$ for many variance profiles as a function of $p/n$. It can be observed that, for all variance profiles considered, the curves mapping $p/n \mapsto \tau_{\min}(\widehat{\Sigma}_n)$ attain their minimum at a value close to zero when $p = n$, corresponding to the interpolation threshold at which the classical double descent phenomenon occurs. This behavior is consistent with previous observations reported in \cite{schaeffer2023double}. A more detailed inspection of \Cref{fig:descent_vp-min}(b), (d) and (f) reveals a notable qualitative difference between variance profiles. In particular, for the piecewise constant and polynomial variance profiles, the smallest non-zero eigenvalue $\tau_{\min}(\widehat{\Sigma}_n)$ remains significantly closer to zero over a wide range of aspect ratios, approximately for $p/n \in [0.25,3]$, compared to the constant or quasi--doubly stochastic cases. This behavior manifests itself as an extended plateau near zero in the corresponding spectral curves. The presence of such plateaus suggests prolonged near-singularity of the empirical covariance matrix across a broad range of over- and under-parameterized regimes. This spectral degeneracy may play a key role in the emergence of more complex risk landscapes, and in particular may be closely related to the appearance of the triple and quadruple descent phenomena observed in \Cref{fig:descent_vp-min}(b). Similar trends are observed for the Alternated columns and Berlin photo variance profiles considered in our experiments. Although it is challenging to visually determine whether this eigenvalue reaches its minimum near the interpolation threshold $p = n$ in \Cref{fig:descent_vp-min}(b), we include zoomed-in versions of this figure (See Figure \ref{fig:Zoom}), which enable a more precise examination of its behavior in the neighborhood of $p/n = 1$.

\begin{figure}[htbp]
\begin{center}
{\subfigure[]{\includegraphics[width = 0.49\textwidth]{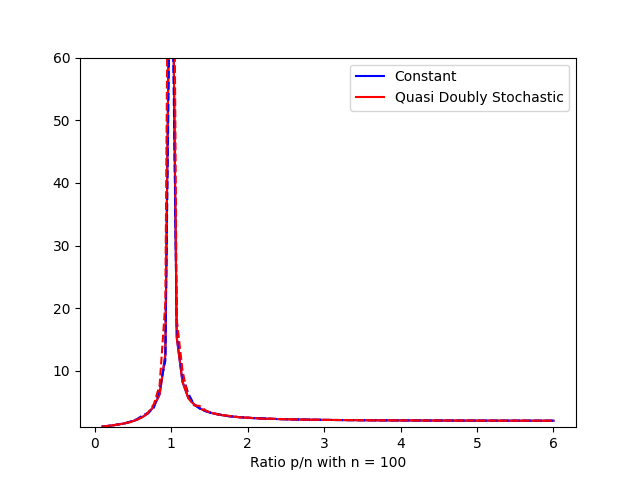}}
\label{fig:db_comparaison_const_db_subfigure}}
\hfill
{\subfigure[]{\includegraphics[width =0.49\textwidth]{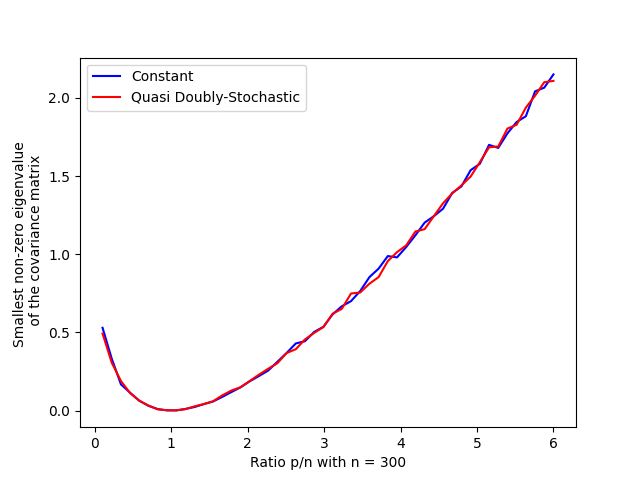}}
\label{vp_min_const_db}}
\\
{\subfigure[]{\includegraphics[width = 0.49\textwidth]{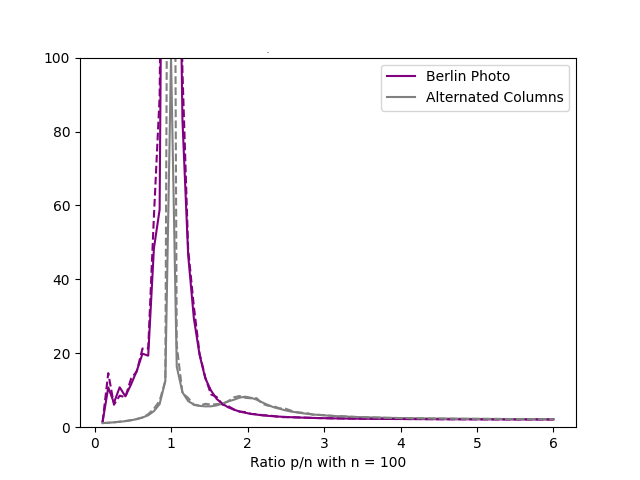}}
\label{fig:db_comparaison_col_berlin}}
\hfill
{\subfigure[]{\includegraphics[width =0.49\textwidth]{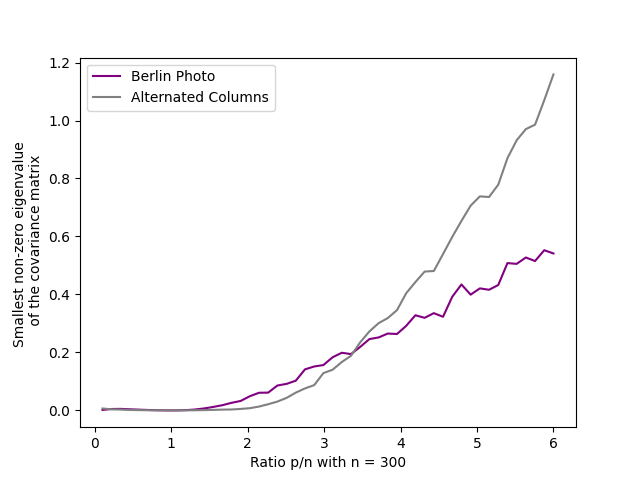}}
\label{fig:vp_min_col_berlin}}
\\
{\subfigure[]{\includegraphics[width = 0.49\textwidth]{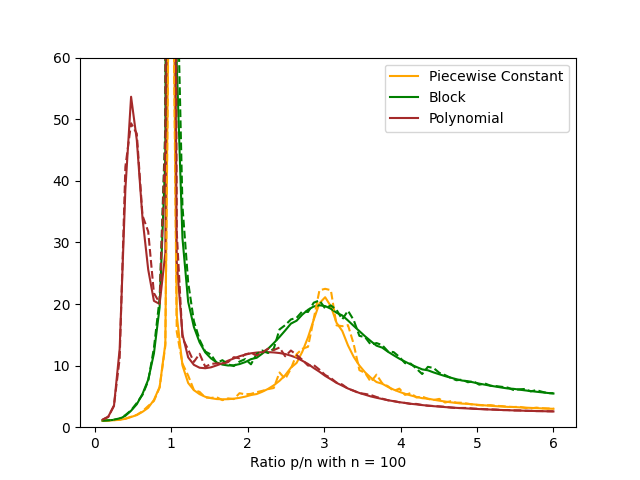}}
\label{fig:db_comparaison_pw_pw3_poly}}
\hfill
{\subfigure[]{\includegraphics[width =0.49\textwidth]{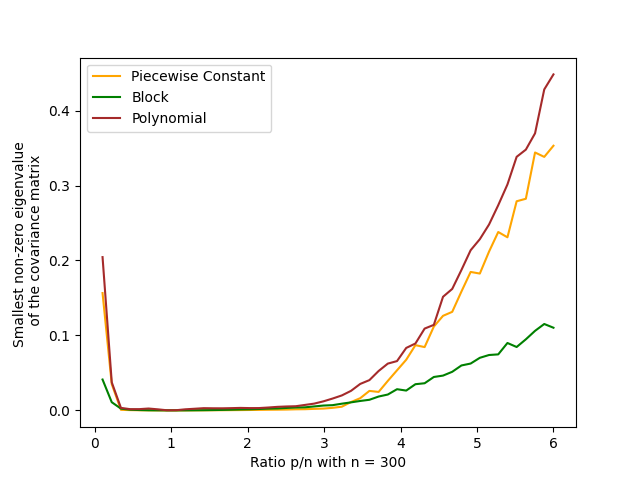}}
\label{fig:vp_min_pw_pw3_poly}}
\end{center}
\caption{Left: Double descent phenomenon for several variance profiles with $\alpha = \sigma = 1$, $n = 100$ and $p$ varying from $10$ to $600$. Right: Smallest non-zero eigenvalue of $\widehat{\Sigma}_n$ for several variance profiles with $\alpha = \sigma = 1$, $n = 300$ and $p$ varying from $30$ to $1800$.} \label{fig:descent_vp-min}
\end{figure}

\begin{figure}[htbp]
\begin{center}
{\subfigure[]{\includegraphics[width = 0.49\textwidth]{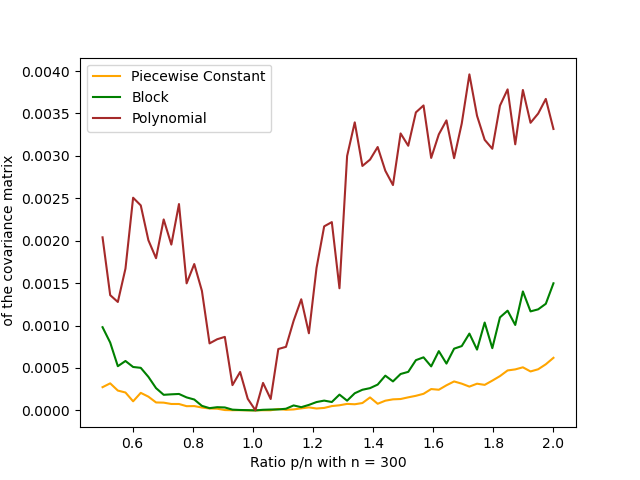}}
\label{fig:zoom1}}
\hfill
{\subfigure[]{\includegraphics[width =0.49\textwidth]{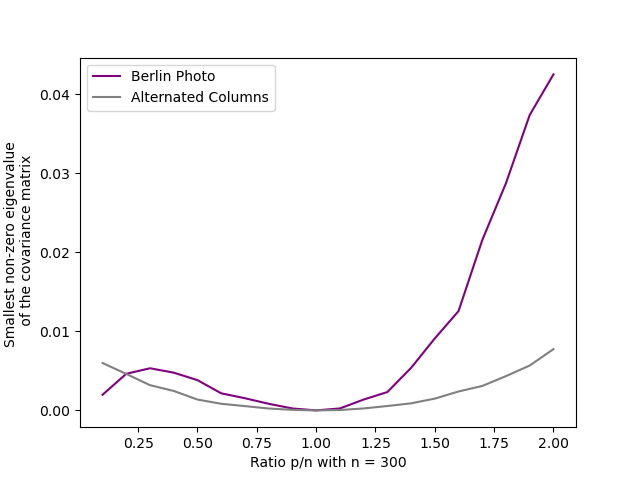}}
\label{fig:zoom2}}
\end{center}
\caption{Smallest non-zero eigenvalue of $\widehat{\Sigma}_n$ for several variance profiles with $\alpha = \sigma = 1$, $n = 300$ and $p$ varying from $30$ to $1800$. Left: Zoomed version of \Cref{fig:descent_vp-min} (f). Right: Zoomed version of \Cref{fig:descent_vp-min} (d).} \label{fig:Zoom}
\end{figure}

\begin{figure}[htbp]
\begin{center}
{\subfigure[]{\includegraphics[width = 0.49\textwidth]{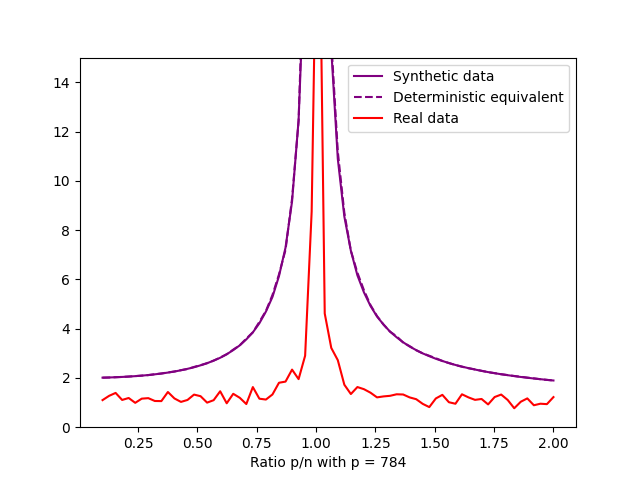}}
\label{fig:db_MNIST}}
\hfill
{\subfigure[]{\includegraphics[width =0.49\textwidth]{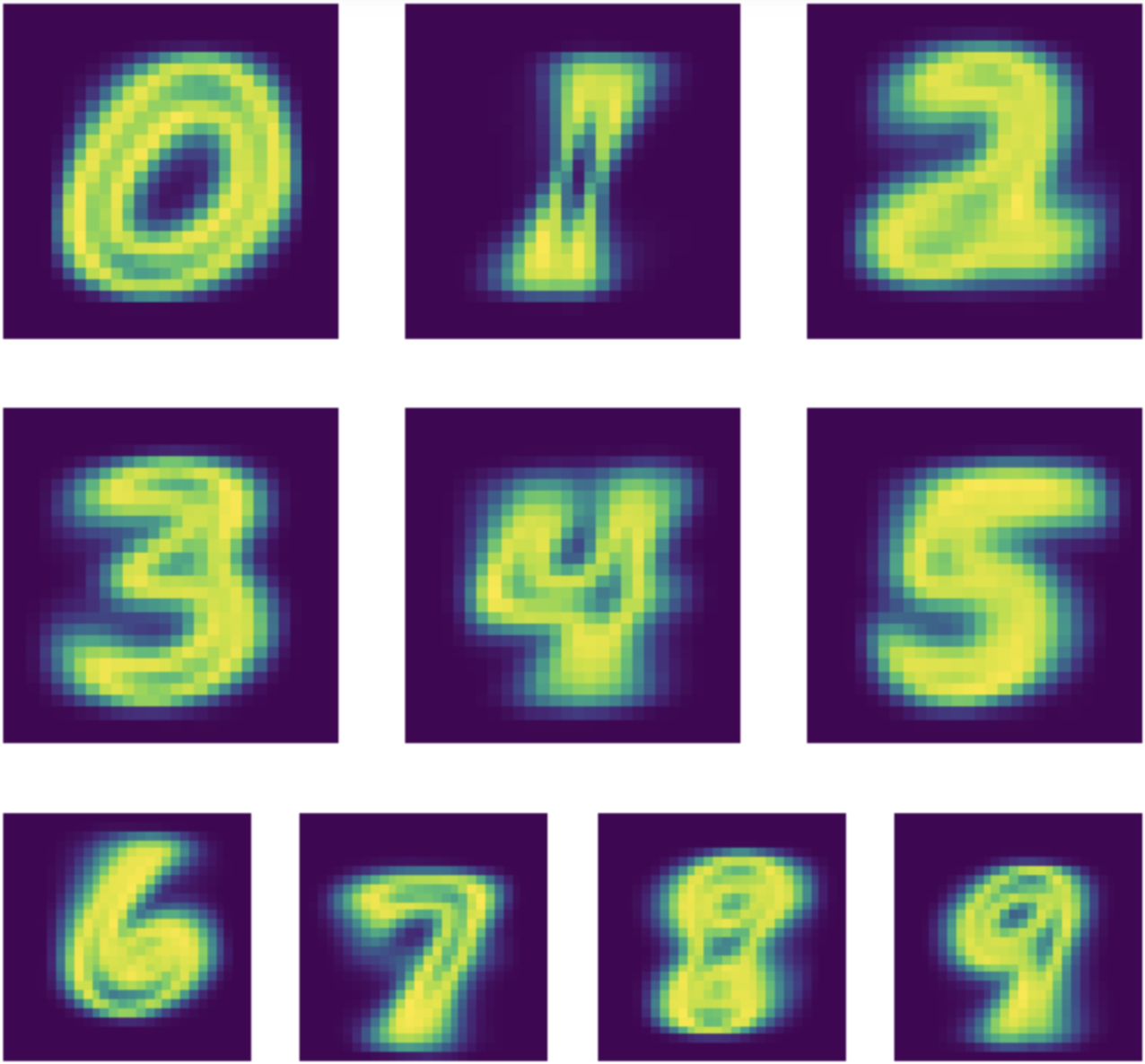}}
\label{fig:Fig_MNIST}}
\end{center}
\caption{Study of MNIST dataset whenever $\lambda = 0$.  (a) Comparison of real and synthetic data with $\alpha = \sigma = 1$, $p = 784$ and $n$ varying from $78$ to $1568$.  The solid lines correspond to the predictive risk while the dashed line corresponds to the deterministic equivalent.(b) Heatmap of the variance of pixels for each class. These heatmaps serve as variance profiles in the mixture model created from the MNIST dataset and described in the begining \Cref{sec:num}.} \label{fig:profile_MNIST}
\end{figure}

 We also conducted numerical experiments that highlight the connection between variance-profiled data and mixture models. These experiments are based on the MNIST dataset, a well-known image classification dataset comprising handwritten digits from 0 to 9. This dataset can indeed be viewed as a 10-class mixture model, as described in the introduction, where the $ i $-th class corresponds to the digit $ i $, and the probability $ \pi_i $ represents the proportion of digit $ i $ in the dataset. Each class exhibits a distinct variance profile, as illustrated in \cref{fig:profile_MNIST} (b). In this figure, we present heatmaps for each class that represent the variance of each pixel within the respective digit images. These heatmaps vary significantly from one class to another, which motivates the use of the variance profile approach. By vectorizing these heatmaps, we extract the variance profiles $ S_i $. Using these profiles, we then examine the regression model of interest through two distinct approaches. In the first approach, we use real data, where the images from the MNIST dataset serve as feature vectors. Since the images are matrices of size $ 28 \times 28 $, we vectorize them into vectors of size 784. In the second approach, we rely on synthetic data by generating random vectors that follow a variance profile $ S_i $ among those described by the heatmaps in \cref{fig:profile_MNIST} (b). For both scenarios, we consider the case where $ \lambda = 0 $ and compare their predictive risks with the deterministic equivalent derived in our theoretical framework. The results of this comparison are summarized in \cref{fig:profile_MNIST} (a). Notably, the curve corresponding to the deterministic equivalent closely matches the one for synthetic data, demonstrating the accuracy of our theoretical predictions. However, these two curves deviate significantly from the one associated with real data. This discrepancy arises because the entries of the feature vectors from real data correspond to the pixels of an image, which are inherently correlated. This violates the assumption of independence among vector entries, a fundamental premise of our study. This limitation underscores one of the primary constraints of our work: we assume independence among the entries of feature vectors, which may be restrictive. Nevertheless, we remain optimistic that our findings can be generalized to scenarios involving correlated data. This optimism is supported by research such as \cite{wagner2012large,benaych2016spectral}, which propose deterministic equivalents for the resolvent matrix in the context of variance-profiled matrices  for correlated Gaussian data. Extending our work in this direction could significantly enhance its applicability to real-world datasets.

\section{Conclusion}  \label{sec:conc}

In this paper, we derive a deterministic equivalent of the DOF and the predictive risk  of ridge (less) regression in a high-dimensional framework with a variance profile to handle the setting of non-iid data. We use RMT to determine a deterministic equivalent of the diagonal  of the resolvent matrix $Q_p(z)$.  Our work extends the study of the DOF and the predictive risk from \cite{17-AOS1549} and \cite{bach2023highdimensional} to the case of general variance profile. It appears that also the result of \cite{17-AOS1549} and \cite{bach2023highdimensional} still hold in the case of a quasi doubly stochastic profile. However, there are variance profiles that cause a behavior of the predictive risk of the minimum norm least square estimator that differs from the double descent phenomenon as classically observed in the case of a constant variance profile. Our numerical experiments confirm that our deterministic equivalent accurately estimates the predictive risk in high-dimension. Our results also allow to understand how assuming such a variance profile for the data influences the statistical properties of ridge regression when compared to the standard assumption of iid observations. 
Beyond the random-matrix analysis itself, the deterministic equivalents provided here can be viewed as a practical ``calculus'' for high-dimensional prediction under heterogeneous designs: they yield computable approximations for prediction error and degrees of freedom, enable principled tuning rules (oracle and GCV-type), and explain when interpolation phenomena resemble or depart from the iid benchmark depending on the variance profile.
These consequences are precisely the type of statistical insights that deterministic equivalents have enabled in the iid ridge literature, and our work extends this program to non-identically distributed designs.

We hope that our approach on the use of  variance profiles may lead to further research works on the statistical analysis of other estimators than ridge regression in more complex models with non-iid data.

\appendix
\section*{Appendix}  \label{sec:app}

\renewcommand{\thesection}{\Alph{section}}
\renewcommand{\theequation}{\thesection.\arabic{equation}}
\setcounter{section}{1}
\setcounter{equation}{0}

In this appendix, we give all the proofs of the results of the paper. We also discuss the link between random matrices with a variance profile and the notion of $\mathcal{R}$-transform in free probability.

\subsection{Proofs of \cref{prop:DOF} and \cref{lem:decomp}}\label{sec:additionalproofs}

\begin{proof}[Proof of \cref{prop:DOF}]
Recall that $ \widehat{df}_{1}(\lambda) = \frac{1}{p} \Tr [\widehat{\Sigma}_n (\widehat{\Sigma}_n + \lambda I_p)^{-1}]$, for $\lambda > 0$. Thus by using $\hat{\mu}_n$ the empirical eigenvalue distribution of $\widehat{\Sigma}_n$ we have this new expression of the degree of freedom $
 \widehat{df}_{1}(\lambda) = \int_{\RR} \frac{t}{t + \lambda} d \hat{\mu}_n(t).$
Moreover, $\hat{\mu}_n$ being a probability measure one can notice from \cref{def:Stieltjes} that 
$
 \widehat{df}_{1}(\lambda) = \int_{\RR} 1 - \frac{\lambda}{t + \lambda} d \hat{\mu}_n(t) = 1 - \lambda g_{\hat{\mu}_n}(-\lambda).
$
Thanks to \cref{prop:Stieltjes}, one has that $ \widehat{df}_{1}(\lambda) \sim 1 - \lambda g_{\nu_n}(-\lambda)$.
Let us now remark that, by \cref{prop:Stieltjes}, one has that
$
 g_{\nu_n}(z) = \frac{1}{p} \Tr [T_p(z)] =  - \frac{1}{p} \sum_{j=1}^{p} \frac{1}{z\left(1 + (1/n) \Tr[\widetilde{D}_j \widetilde{T}_n(z)]\right)},
$
implying that
\begin{eqnarray*}
1 - \lambda g_{\nu_n}(- \lambda)   =  1 - \frac{1}{p} \sum_{j=1}^{p} \frac{1}{\left(1 + (1/n) \Tr[\widetilde{D}_j \widetilde{T}_n(- \lambda)]\right)}=    \frac{1}{p} \sum_{j=1}^{p} \frac{ (1/n) \Tr[\widetilde{D}_j \widetilde{T}_n(- \lambda)]}{\left(1 + (1/n) \Tr[\widetilde{D}_j \widetilde{T}_n(- \lambda)]\right)} .
\end{eqnarray*}
Then, recall that
$
\Sigma_n =   \frac{1}{n}  \diag( \Tr[\widetilde{D}_1],\ldots,  \Tr[\widetilde{D}_p]),
$
which leads to
\begin{equation}
1 - \lambda g_{\nu_n}(- \lambda) =  \frac{1}{p} \sum_{j=1}^{p} \frac{ (1/n) \Tr[\widetilde{D}_j]  }{ (1/n) \Tr[\widetilde{D}_j] + \kappa_{j}(\lambda)} =   \frac{1}{p}  \Tr[ \Sigma_n ( \Sigma_n + \kappa(\lambda))^{-1} ],  \label{eq:equiv2}
\end{equation}
where
$\kappa_{j}(\lambda) =  \frac{ \Tr[\widetilde{D}_j]}{  \Tr[\widetilde{D}_j \widetilde{T}_n(- \lambda)]}$, for $1 \leq j \leq p,
$
and 
$
\kappa(\lambda) = \diag(\kappa_{1}(\lambda),\ldots, \kappa_{p}(\lambda) ).
$

Consequently, we obtain the deterministic equivalent for the DOF stated in \cref{eq:DOFprop}, which completes the proof of \cref{prop:DOF}.
\end{proof}

\begin{proof}[Proof of \cref{lem:decomp}]
Let's compute a new expression for $\hat{r}_\lambda^{train} (X_n)$
\begin{eqnarray*}
\hat{r}_\lambda^{train} (X_n) &=& \frac{1}{n}\mathbb{E} [\parallel Y_n - X_n \hat{\theta}_{\lambda} \parallel^2|X_n]= \frac{1}{n}\mathbb{E} [\parallel (I_n - \frac{1}{n}X_n X_n^\top \widetilde Q_n(-\lambda)) Y_n \parallel^2|X_n]\\
                   &=& \frac{\lambda^2}{n}\mathbb{E} [\parallel \widetilde Q_n(-\lambda) Y_n \parallel^2|X_n]\\
                   &=& \frac{\lambda^2}{n}\mathbb{E} [\Tr[\widetilde Q_n(-\lambda)X\beta_\ast\beta_\ast^\top X_n^\top \widetilde Q_n(-\lambda) +2\widetilde Q_n(-\lambda)X_n\beta_\ast\varepsilon^\top \widetilde Q_n(-\lambda) \\
                   & & + \widetilde Q_n(-\lambda)\varepsilon\varepsilon^\top \widetilde Q_n(-\lambda)]|X_n]\\ &=&\frac{\lambda^2\alpha^2}{p}\mathrm{Tr}\Bigg[\widetilde Q_n^2(-\lambda)\frac{X_nX_n^\top}{n}\Bigg] + \frac{\lambda^2\sigma^2}{n}\mathrm{Tr}[\widetilde Q_n^2(-\lambda)] \\ %&=&\frac{\lambda^2\alpha^2}{n}\mathrm{Tr}\Bigg[Q_p(-\lambda) - \lambda Q^2_p(-\lambda)\Bigg] + \frac{\lambda^2\sigma^2}{n}\mathrm{Tr}[Q^2_p(-\lambda)]\\
                  & = &\frac{\lambda^2\alpha^2}{p}\mathrm{Tr}\Bigg[\widetilde Q_n(-\lambda) - \widetilde Q_n'(-\lambda)\Bigg] + \frac{\lambda^2\sigma^2}{n}\mathrm{Tr}[\widetilde Q_n'(-\lambda)].
\end{eqnarray*}

Let us consider the following decomposition of the predictive risk. Since $\tilde{x}$ and $\widetilde{\varepsilon}$ are independent from $X_n$, we can proceed similarly to get a new expression for $\hat{r}_\lambda^{test} (X_n)$:
\begin{eqnarray*}
\hat{r}_\lambda^{test} (X_n)   &=& \mathbb{E}[( \tilde{x}^\top \hat{\theta}_\lambda - \tilde{x}^\top \beta_\ast - \tilde{\varepsilon} )^2| X_n]\\
			&=& \mathbb{E}[\mathrm{Tr}[\hat{\theta}_\lambda\hat{\theta}_\lambda^\top  \tilde{x} \tilde{x}^\top  + \beta_\ast\beta_\ast^\top  \tilde{x} \tilde{x}^\top  + \tilde{\varepsilon}^2 - 2\beta_\ast\hat{\theta}_\lambda^\top  \tilde{x} \tilde{x}^\top  - 2\tilde{\varepsilon}\hat{\theta}_\lambda^\top  \tilde{x}\\
			& &+ 2\tilde{\varepsilon}\beta_\ast^\top  \tilde{x}
			]|X_n]\\
			&=& \sigma^2 +  \mathrm{Tr}[\mathbb{E}[(\hat{\theta}_\lambda\hat{\theta}_\lambda^\top  + \beta_\ast\beta_\ast^\top  - 2\beta_\ast\hat{\theta}_\lambda^\top )|X_n]\widetilde{S}_p ]\\
			&=& \sigma^2 +  \mathbb{E}[\mathrm{Tr}[(\hat{\theta}_{\lambda} - \beta_\ast) (\hat{\theta}_{\lambda} - \beta_\ast)^\top \widetilde{S}_p ]|X_n]\\
			&=& \sigma^2 +  \mathbb{E}[(\hat{\theta}_{\lambda} - \beta_\ast)^\top \widetilde{S}_p (\hat{\theta}_{\lambda} - \beta_\ast) |X_n] = \sigma^2 + \hat R(\hat{\theta}_{\lambda},\beta_\ast),
\end{eqnarray*}
with $\hat R(\hat{\theta}_{\lambda},\beta_\ast) = \mathbb{E}[(\hat{\theta}_{\lambda} - \beta_\ast)^\top \widetilde{S}_p (\hat{\theta}_{\lambda} - \beta_\ast) |X_n]$.
Then by using \cref{eq:linmod,eq:theta}, we get $\hat{\theta}_{\lambda} - \beta_\ast =  -\lambda (\widehat{\Sigma}_n +  \lambda I_p)^{-1} \beta_\ast  + (\widehat{\Sigma}_n +  \lambda I_p)^{-1} \frac{X_n^\top  \varepsilon_n}{n}.$ Therefore, we can rewrite
\begin{eqnarray*}
\hat{R}(\hat{\theta}_{\lambda},\beta_\ast) &=& \mathbb{E}[\lambda^2 \beta_\ast^\top  Q_p(-\lambda) \widetilde{S}_p Q_p(-\lambda) \beta_\ast + \frac{\varepsilon_n^\top X_n}{n}Q_p(-\lambda) \widetilde{S}_p Q_p(-\lambda)\frac{X_n^\top  \varepsilon_n}{n} \\
                & & -2 \lambda \beta_\ast^\top  Q_p(-\lambda) \widetilde{S}_p Q_p(-\lambda)\frac{X_n^\top  \varepsilon_n}{n} |X_n ]\\
                &=& \lambda^2 \mathrm{Tr}[Q_p(-\lambda) \widetilde{S}_p Q_p(-\lambda) \mathbb E [\beta_\ast\beta_\ast^\top ]] + \frac{1}{n^2}\mathrm{Tr}[\widetilde{S}_p Q_p(-\lambda) X_n^\top \mathbb{E}[\varepsilon_n\varepsilon_n^\top ]X_nQ_p(-\lambda)]\\
            & = & \frac{\lambda^2\alpha^2}{p} \mathrm{Tr}[Q_p(-\lambda) \widetilde{S}_p Q_p(-\lambda)] + \frac{\sigma^2}{n} \Tr[\widetilde{S}_p   Q_p(-\lambda)\widehat{\Sigma}_n   Q_p(-\lambda) ] \\
                & = & \frac{\lambda^2\alpha^2}{p} \mathrm{Tr}[Q_p(-\lambda) \widetilde{S}_p Q_p(-\lambda)] + \frac{\sigma^2}{n} \Tr[\widetilde{S}_p   \widehat{\Sigma}_n   Q_p'(-\lambda) ].
\end{eqnarray*}
Note that $\widehat{\Sigma}_n$ and $Q_p(-\lambda)$ are commuting, one has that $Q_p^2(-\lambda) = Q_p'(-\lambda)$, and moreover, $\widehat{\Sigma}_n Q_p'(-\lambda) = Q_p(-\lambda) - \lambda Q_p'(\lambda)$. Then we get the following Bias-Variance decomposition of the predictive risk
$
\hat{R}(\hat{\theta}_{\lambda},\beta_\ast)=\mathrm{Bias}(\hat{\theta}_{\lambda}) + \mathrm{Var}(\hat{\theta}_{\lambda}),
$
where
$\mathop{\rm Bias}\nolimits(\hat{\theta}_{\lambda}) = \frac{\lambda^2\alpha^2}{p} \mathrm{Tr}[Q_p(-\lambda) \widetilde{S}_p Q_p(-\lambda)]  =  \frac{\lambda^2\alpha^2}{p} \mathrm{Tr} [\widetilde{S}_p Q_p'(-\lambda)]$and
$ \mathop{\rm Var}(\hat{\theta}_{\lambda}) =  \frac{\sigma^2}{n} \Tr[\widetilde{S}_p   \widehat{\Sigma}_n   Q_p'(-\lambda) ] =  \frac{\sigma^2}{n} \mathrm{Tr}[\widetilde{S}_p(Q_p(-\lambda) -\lambda Q_p'(-\lambda)) ].
$
From these computations, we finally deduce the following formula of the predictive risk
\begin{eqnarray*}
    \hat{r}_\lambda^{test} (X_n) = \sigma^2 + \frac{\sigma^2}{n}\mathrm{Tr}[\widetilde{S}_p Q_p(-\lambda)] + \lambda\left( \frac{\lambda\alpha^2}{p} -  \frac{\sigma^2}{n} \right) \mathrm{Tr}[\widetilde{S}_p  Q_p'(-\lambda)],
\end{eqnarray*}
Then, \cref{eq:risk} directly follows from these computations since $\widetilde{S}_p$ is a diagonal matrix, which completes the proof of \cref{lem:decomp}.
\end{proof}

\subsection{Derivation of the deterministic equivalent the diagonal of the resolvent}\label{sec:proofs}
This section is dedicated to prove \cref{theoremRMT} that aims to provide a deterministic equivalent of $\Delta[Q_p(z)]$ that is necessary to obtain asymptotic equivalents of the risks.
The proof of this theorem is strongly inspired by results from \cite{HLJ07}, and it constitutes a key element in the proof of \cref{theoremStat}. In order to prove \cref{theoremRMT}, we introduce the diagonal matrices $R_p(z) = \underset{1\leq j \leq p}{\diag} (R^{(j)}_p(z))$ defined by
$
R^{(j)}_p(z) = \frac{-1}{z\left((1 + (1/n) \Tr[\widetilde{D}_j \widetilde{Q}_n(z)]\right)} \quad \mbox{for } z \in \mathbb C \setminus \mathbb R^+,
$
with $\widetilde{Q}_n(z) = \Big(\widetilde{\Sigma}_n - zI_n\Big)^{-1}$ is the resolvent matrix of $\widetilde{\Sigma}_n = \frac{1}{n} X_n X_n^\top $.
\begin{lemma}\label{lem:Q-R}
Let us define $\Delta[\mathbf{Q}(z)]$ as in \cref{theoremRMT} and denote by $\mathbf{R}(z)$ the family of matrices  $(R_p(z))_{p\geq 1}$. Then one has that 
$
\Delta[\mathbf{Q}(z)] \sim \mathbf{R}(z),
$
for all $z\in \mathbb C ^+ = \lbrace z\in \mathbb C, \Im(z) > 0 \rbrace$.
\end{lemma}

\begin{proof}[Proof of \cref{lem:Q-R}]
This proof is based on [2][Lemma 6.1] which asserts that there exists a constant $K_1 > 0$ such that $\mathbb{E}\left[ \left| \frac{1}{p}\Tr[(Q_p(z)-R_p(z))U_p]\right|^{2+\delta/2}\right]\leq \frac{K_1}{p^{1+\delta/4}}$, where $(U_p)_{p \geq 1}$ denotes a family of deterministic matrices satisfying \cref{eq:diag_family}. We deduce from this inequality that 
\begin{eqnarray*}
    \mathbb{E}\left[ \sum_{p\geq 1} \left| \frac{1}{p}\Tr[(Q_p(z)-R_p(z))U_p] \right|^{2+\delta/2}\right] \leq \sum_{p\geq 1}\frac{K_1}{p^{1+\delta/4}} < +\infty.
\end{eqnarray*}
The series being finite, one has that 
$
\mathbb{P}\left[ \sum\limits_{p\geq 1} \left| \frac{1}{p}\Tr[(Q_p(z)-R_p(z))U_p] \right|^{2+\varepsilon/2} < +\infty \right] = 1
$
, which finally gives us: 
\begin{eqnarray*}
\mathbb{P}\left[  \frac{1}{p}\Tr[(Q_p(z)-R_p(z))U_p]   \xrightarrow[n\rightarrow +\infty]{} 0 \right]\geq \mathbb{P}\left[ \sum_{p\geq 1} \left| \frac{1}{p}\Tr[(Q_p(z)-R_p(z))U_{p}] \right|^{2+\varepsilon/2} < +\infty \right]= 1.
\end{eqnarray*}
This concludes the proof of \cref{lem:Q-R}.
\end{proof}

Then, \cref{lem:Q-R} needs to be combined with the following one in order to prove Theorem \cref{theoremRMT}.
\begin{lemma}\label{lem:T-R}
Let us define $\mathbf{T}(z)$ and $\mathbf{R}(z)$ as in \cref{theoremRMT} and \cref{lem:Q-R}. Then one has that $
\mathbf{R}(z) \sim \mathbf{T}(z),
$
for all $z\in \mathcal{D} = \lbrace z\in \mathbb C^+, \frac{|z|}{|\Im(z)|} < 2, |\Im(z)| > 4d\gamma_{max}^2 \rbrace$.
\end{lemma}

\begin{proof}[Proof of \cref{lem:T-R}]
Since $T_p(z)$ and $R_p(z)$ are diagonal matrices, the following equations hold true $
\mathrm{Tr}[(R_p(z)-T_p(z))U_p] = \sum_{i=1}^p (R_p^{(i)}(z)-T_{p}^{(i)}(z))U_p^{(i)}$,
 where $(U_p)_{p \geq 1}$ denotes a family of deterministic matrices satisfying \cref{eq:diag_family}, and
$$(R_p^{(i)}(z)-T_{p}^{(i)}(z)) = R_p^{(i)}(z)T_{p}^{(i)}(z)(1/ R_p^{(i)}(z)-1/T_{p}^{(i)}(z)).$$
We know from  [2][Proposition 5.1], that $|R_p^{(i)}(z)T_{p}^{(i)}(z)| \leq \frac{1}{|\Im(z)|^2}$. \\Moreover [2][Equation (6.15)] states that
$
    \sup\limits_{1\leq i\leq p}\mathbb{E}[|(1/ R_p^{(i)}(z)-1/T_{p}^{(i)}(z))|^{2+\delta/2}] \leq \frac{K_2}{n^{1+\delta/4}}.
$
 Therefore, we obtain from these inequalities and \cref{eq:diag_family} that
 \begin{eqnarray*}
\mathbb{E}\left[ \sum_{p\geq 1} \left| \frac{1}{p}\Tr[(T_p(z)-R_p(z))U_p] \right|^{2+\delta/2}\right] &\leq&  \sum_{p\geq 1}  \frac{1}{p}\sum_{i =1}^p |U_p^{(i)}|^{2+\delta / 2} \mathbb{E}\left[|T_{p}^{(i)}(z) - R_p^{(i)}(z)|^{2+\delta / 2}\right]\\
&\leq&  \sum_{p\geq 1}\frac{K^{2+\delta/2}K_2}{|\Im (z)|^2p^{1+\delta/4}}
\leq \sum_{p\geq 1} \frac{\widetilde K}{p^{1+\delta/4}}
 <  + \infty, 
\end{eqnarray*}
 with $\widetilde{K} = \frac{K^{2+\delta/2}K_2}{16d^2\gamma_{max}^4}$, since $z \in \mathcal{D}$.
 Therefore, arguing as in the  proof of \cref{lem:Q-R}, we obtain that 
 $
 \mathbb{P}\left[  \frac{1}{p}\Tr[(T_p(z)-R_p(z))U_p]   \xrightarrow[n\rightarrow +\infty]{} 0 \right]  = 1,
 $
and this concludes the proof of \cref{lem:T-R}.
\end{proof}

\begin{proof}[Proof of \cref{theoremRMT}]
    Since the equivalence relation introduced in \cref{def:equiv} is transitive, we deduce from \cref{lem:Q-R,lem:T-R} that$
        \frac{1}{p} \Tr[Q_p(z)U_p] \sim \frac{1}{p} \Tr[T_p(z)U_p]$,
    for $z\in \mathcal{D}$.
    It remains to prove that this equivalence is true for $z\in \mathbb C \setminus \mathbb R^+$. To this end, let us define the  function
    $f_p(z) =\frac{1}{p} \Tr[Q_p(z)U_p] - \frac{1}{p} \Tr[T_p(z)U_p]$ for $z\in \mathbb C \setminus \mathbb R^+$. Proving $
        \frac{1}{p} \Tr[Q_p(z)U_p] \sim \frac{1}{p} \Tr[T_p(z)U_p]$ is equivalent to prove that $f_p(z)$ converges uniformly to zero. This function is analytic on $z\in \mathbb C \setminus \mathbb R^+$ since $Q_p(z)$ and $T_p(z)$ are analytic \cite{HLJ07}[Proposition 5.1]. For $1\leq i \leq p$, $T_p^{(i)} (z)$ is a Stieltjes transforms \cite{HLJ07}[Theorem 2.4], then we can state that $|T_p^{(i)}(z)|\leq \frac{1}{\mathrm{dist}(z,\mathbb{R}^+)}$ with $z \in \mathbb C \setminus \mathbb R^+$, according to \cite{HLJ07}[Proposition 2.2]. Moreover, we deduce from \cite{HLJ07}[Proposition 2.3] and \cite{hislop2012introduction}[Theorem 5.8] that $|Q_{p}^{(i)}(z)| \leq \parallel Q_p(z) \parallel_{sp} \leq \frac{1}{\mathrm{dist}(z,\mathbb{R}^+)}$, where $\parallel . \parallel_{sp}$ is the spectral norm, $Q_{p}^{(i)}(z)$ denotes the i-th diagonal entry of $Q_p(z)$ and $z \in \mathbb C \setminus \mathbb R^+$.
    Thus, using \cref{eq:diag_family}, we have proved   that
    \begin{equation}\label{ineq:QT-U}
        |\frac{1}{p} \Tr[Q_p(z)U_p]| \leq \frac{K}{\mathrm{dist}(z,\mathbb{R}^+)} \quad \mathrm{and} \quad |\frac{1}{p} \Tr[T_p(z)U_p]| \leq \frac{K}{\mathrm{dist}(z,\mathbb{R}^+)}.
    \end{equation}
 Hence, for each compact subset $C \subset \mathbb C \setminus \mathbb R^+$, $f_p$ is uniformly bounded on $C$, that is $|f_p(z)| \leq \frac{2K}{\delta_C}$,
    where $\delta_C$ is the distance between $C$ and $\mathbb R^+$. Then, by the normal family Theorem \cite{rudin1987real}[Theorem 14.6] there exists a sub-sequence $f_{p_k}$ which uniformly converges  to $f^\ast$ that is an analytical function on $\mathbb C \setminus \mathbb R^+$.
    Let $(z_k)_{k\in\mathbb N}$ be a sequence with an accumulation point in $\mathcal{D}$. Then, for each $k$, $f_p(z_k) \rightarrow 0$ with probability one as $p$ tends to $+\infty$. This implies that $f^\ast (z_k) = 0$ for each $k$. We finally obtain that $f^\ast$ is identically zero on $\mathbb C \setminus \mathbb R^+$. Therefore $f_p$ converges uniformly to zero which proves that $
        \frac{1}{p} \Tr[Q_p(z)U_p] \sim \frac{1}{p} \Tr[T_p(z)U_p]$ for $z \in \mathbb C \setminus \mathbb{R}^+$.

    Now, it remains to prove that$
        \frac{1}{p} \Tr[Q_p'(z)U_p] \sim \frac{1}{p} \Tr[T'(z)U_p]$.
    The functions $Q_p(z)$ and $T_p(z)$ are analytic on  $\mathbb C \setminus \mathbb R^+$ (see \cite{hislop2012introduction}[Theorem 1.2] and \cite{HLJ07}[Proposition 5.1]), then the Cauchy integral formula yields that 
    \begin{eqnarray*}
T_p'(z) = \frac{1}{2\pi i} \int_\rho \frac{T_p(w)}{(w-z)^2} dw \quad \mathrm{and} \quad  Q_p'(z) = \frac{1}{2\pi i} \int_\rho \frac{Q_p(w)}{(w-z)^2} dw,
\end{eqnarray*}
    where $\rho$ is a path around $z$ in $\mathbb C \setminus \mathbb R^+$.
Hence, we have that 
    \begin{eqnarray*}
        \frac{1}{p} \Tr[Q_p'(z)U_p] - \frac{1}{p} \Tr[T_p'(z)U_p] &=& \frac{1}{p}\mathrm{Tr}\Bigg[\int_\rho \frac{Q_p(w)-T_p(w)}{(w-z)^2} dw \ U_p\Bigg]= \int_\rho \frac{1}{p}\frac{\mathrm{Tr}[(Q_p(w)-T_p(w))U_p]}{(w-z)^2} dw.
    \end{eqnarray*} 
Since we have proved that $
        \frac{1}{p} \Tr[Q_p(z)U_p] \sim \frac{1}{p} \Tr[T_p(z)U_p]$ holds on $\mathbb C \setminus \mathbb R^+$, we obtain that for all $w\in \rho$ that $
\lim_{n \to \infty, \; p/n \to c}\frac{1}{p}\mathrm{Tr}[(Q_p(w)-T_p(w))U_p] = 0.
$
 Moreover, for $w\in \mathbb C \setminus \mathbb R^+$, we have from Equation \cref{ineq:QT-U} that
$
     \frac{1}{p}\frac{|\mathrm{Tr}[(Q_p(w)-T_p(w))U_p]|}{|w-z|^2}
     \leq \frac{2K}{\mathrm{dist}(z,\mathbb{R}^+)|w-z|^2}
     \leq \frac{2K}{\mathrm{dist}(\rho,\mathbb{R}^+)|w-z|^2}.
$ 
 Since $\frac{2K}{\mathrm{dist}(\rho,\mathbb{R}^+)|w-z|^2}$ is integrable on $\rho$, using the dominated convergence theorem, we obtain that $\frac{1}{p} \Tr[Q_p'(z)U_p] \sim \frac{1}{p} \Tr[T'(z)U_p]$ holds for $z \in \mathbb C \setminus \mathbb R^+$. 
 
 Moreover, one can use the same proof to demonstrate that $
        \frac{1}{n} \Tr[\widetilde Q_n(z)U_n] \sim \frac{1}{n} \Tr[\widetilde T_n(z)U_n]$ and $\frac{1}{n} \Tr[\widetilde Q_n'(z)U_n] \sim \frac{1}{n} \Tr[\widetilde T_n'(z)U_n]$ for $z \in \mathbb C \setminus \mathbb R^+$ and $U_n \in\mathbb R^{n\times n}$ determinsstic, which completes the proof of  \cref{theoremRMT}.
\end{proof}

\subsection{Proof of the main results} \label{sec:mainproofs}

We have now all the ingredients needed to prove \cref{theoremStat}.

\begin{proof}[Proof of \cref{theoremStat}]
    Since \cref{hyp:var1} hold true for the variance profile $\widetilde{S}_p$ and $\Upsilon_n$, it follows that $(\widetilde{S}_p)_{p\geq 1}$ and $(\widehat{\Sigma}_n)_{n\geq 1}$ satisfy \cref{eq:diag_family}. Moreover the family of identity matrices, namely $(I_p)_{p\geq 1}$, also satisfies \cref{eq:diag_family} thus, we deduce from \cref{theoremRMT} that
\begin{equation}\label{eq:equiv_proof-1}
\frac{1}{p} \Tr[Q_p(z)\widetilde{S}_p] \sim \frac{1}{p} \Tr[T_p(z)\widetilde{S}_p]  \mbox{ , }  \frac{1}{p} \Tr[Q_p'(z)\widetilde{S}_p] \sim \frac{1}{p} \Tr[T'(z)\widetilde{S}_p],
\end{equation}
\begin{equation}\label{eq:equiv_proof-2}
\frac{1}{p} \Tr[Q_p'(z)\widehat{\Sigma}_n] \sim \frac{1}{p} \Tr[T_p'(z)\widehat{\Sigma}_n] \mbox{ , } \frac{1}{p} \Tr[Q_p'(z)] \sim \frac{1}{p} \Tr[T'(z)],  
\end{equation}
\begin{equation}\label{eq:equiv_proof-3}
\frac{1}{n} \Tr[\widetilde Q_n(z)] \sim \frac{1}{n} \Tr[\widetilde T_n(z)] \mbox{ , } \frac{1}{n} \Tr[\widetilde Q_n'(z)] \sim \frac{1}{n} \Tr[\widetilde T_n'(z)].  
\end{equation}
Hence \cref{eq:equiv_det} directly follows from Equations \cref{eq:equiv_proof-1,eq:equiv_proof-2,eq:equiv_proof-3} and \cref{eq:risk}.
\end{proof}

\begin{proof}[Proof of \cref{lem:limzero}]
We assume that \cref{hyp:Z,hyp:var1} hold true in all this proof. According to [2][Theorem 2.4], one has that $T_p(-\lambda) = \int_{\mathbb R^+}\frac{\mu (dw)}{w+\lambda}$ for $\lambda > 0$, where $\mu = (\mu_{ij})$ is a positive $p\times p$ matrix valued measure such that $\mu_{ij} (\mathbb R^+) = \delta_{ij}$.
Moreover, $\lambda\mapsto \frac{1}{w/\lambda+1}$ is differentiable for $w\in \mathbb R^+$, its derivative $w\mapsto \frac{w}{(w+\lambda)^2}$ is measurable  for $\lambda > 0$ and $|\frac{w}{(w+\lambda)^2}| \leq 1$, thus $T_p(-\lambda)$ is differentiable and $
    T_p'(-\lambda) = \int_{\mathbb R^+}\frac{\mu (dw)}{(w+\lambda)^2}$.

As stated in \cref{prop:Stieltjes} $\frac{1}{p}\mathrm{Tr}[T_p(z)]$ is a Stieltjes transform and $\frac{1}{p}\mathrm{Tr}[T_p(z)] = \int_{\mathbb R^+}\frac{\nu_p (dw)}{w+\lambda}$,
where there exists $\pi_\ast \in [0,1]$ and $\pi : (0,+\infty) \rightarrow [0,+\infty )$ a locally Hölder-continuous function such that  $\nu_p (dw) = \pi_\ast \delta_0(dw) + \pi (dw) \mathbf{1}_{w>0}dw$ [1][Theorem 2.1].

Under the additional \cref{hyp:var2,hyp:dim_rec}, [1][Theorem 2.9] asserts that there exists $\tau > 0$ such that $\pi ((0,\tau]) = 0$ and if $p < n$ then $\pi_\ast = 0$.
Since $T_p(-\lambda) = \int_{\mathbb R^+}\frac{\mu (dw)}{w+\lambda}$ for $\lambda > 0$, one get that 
$\int_0^\tau \frac{\frac{1}{p}\mathrm{Tr}(\mu(dw))}{w+\lambda} = \int_0^\tau \frac{\nu_p(dw)}{w+\lambda} =  0$.

Since $\frac{1}{w+\lambda} > 0$ for $\lambda > 0$ and $w\geq 0$, we deduce from the previous equation that $\frac{1}{p} \sum\limits_{i=1}^p \mu_{ii}((0,\tau]) = 0$. Which gives us $\mu_{ii}((0,\tau]) = 0$ for $1\leq i\leq p$ because $(\mu_{ii})_{1\leq i \leq p}$ are positive measures. Hence $
    T_p(-\lambda) = \frac{\pi_\ast}{\lambda} + \int_{\tau}^{+\infty}\frac{\mu (dw)}{w+\lambda}$.

Let's focus on the case $p<n$, we then have 
$T_p(-\lambda) = \int_{\tau}^{+\infty}\frac{\mu (dw)}{w+\lambda}$.

Moreover, $|\frac{1}{w+\lambda}| \leq \frac{1}{\tau}$ for $\lambda >0$ and $w \geq \tau$, thus by the dominated convergence theorem $T_p(-\lambda)$ admits a limit when $\lambda$ tends to $0$,  we denote $T_p(0^-) = \int_{ \tau}^{+\infty} \frac{ \mu(dw)}{w}$ this limit. A similar proof allows us to state that $T_p'(-\lambda)$ admits a limits when $\lambda$ tends to $0$, we denote $T'(0^-) = \int_{ \tau}^{+\infty} \frac{ \mu(dw)}{w^2}$ this limit.
\end{proof}

\begin{proof}[Proof of \cref{CorSto}]
As stated in \Cref{sec:varprofile}, in the case of a quasi doubly stochastic variance profile, $T_p(z)$ is a scalar matrix. In particular, $T_p(z) = m_p(z)I_p$ where $m_p(z)$ is the Stieltjes transform of the Marchenko-Pastur distribution. Moreover, we know from \cref{{prop:Stieltjes}} that $\frac{1}{p}\mathrm{Tr}[T_p(z)] $ is a Stieltjes Transform associated to a measure $\nu_p$. Therefore, $\frac{1}{p}\mathrm{Tr}[T_p(z)] = \int_{\mathbb R^+}\frac{\nu_p (dw)}{w+\lambda}$ and
$
m_p(z) = \int_{\mathbb R^+}\frac{\nu_p (dw)}{w+\lambda}.
$
Additionnally, one has from \cite{alt2017local}[Theorem 2.1] that $\nu_p (dw) = \pi_\ast \delta_0(dw) + \pi (dw) \mathbf{1}_{w>0}dw$. We have seen in the previous proof that there exists $\tau > 0$ such that $\nu_p ((0,\tau]) = 0$ and if $p<n$ then $\pi_\ast = 0$. This implies that $0$ belongs to the domain of definition of $m_p(.)$. Hence $\lim\limits_{\lambda \rightarrow 0}T_p(-\lambda) =\lim\limits_{\lambda \rightarrow 0} m_p(-\lambda) = m_p(0)$ and consequently $T_p(0^-) = m_p(0)$.

We prove symmetrically that if $p>n$ then $0$ belongs to the domain of definition of $\tilde{m}_n(.)$. Moreover, we have seen in \Cref{sec:dof} that in the case of a quasi doubly stochastic variance profile $\kappa (z) = \frac{1}{\tilde{m}_n(z)}I_n$. Note that $\tilde{m}_n(0) = \int_{\tilde{\tau}}^{+\infty}\frac{\tilde\nu_n (dw)}{w}$ is positive. Hence we conclude that $\kappa (0^+) = \frac{1}{\tilde{m}_n(0)}I_n$. We prove $\kappa '(0^+) = \frac{\tilde{m}_n'(0)}{\tilde{m}_n^2(0)}I_n$ the same way.
\end{proof}

\begin{proof}[Proof of \cref{CorStat}]
We assume that \cref{hyp:Z,hyp:var1} hold true in all this proof.
We first focus on the limit with $\lambda \rightarrow +\infty$.
According to [2][Theorem 2.4], $T_p(-\lambda)$ can be expressed as follows for $\lambda > 0$
\begin{eqnarray*}
    T_p(-\lambda) = \int_{\mathbb R^+}\frac{\mu (dw)}{w+\lambda} = \frac{1}{\lambda}\int_{\mathbb R^+}\frac{\mu (dw)}{w/\lambda+1},
\end{eqnarray*}
where $\mu = (\mu_{ij})$ is a $p\times p$ matrix valued measure such that $\mu (\mathbb R^+) = I_p$.
Since  $w\mapsto \frac{1}{w/\lambda+1}$ is a measurable function for $\lambda > 0$,$\lim\limits_{\lambda \rightarrow +\infty} \frac{1}{w/\lambda+1} = 1$ for $w\in \mathbb R^+$ and $|\frac{1}{w/\lambda+1}| \leq 1$, we deduce from the dominated convergence theorem that 
$
\int_{\mathbb R^+}\frac{\mu (dw)}{w/\lambda+1} \xrightarrow[\lambda \rightarrow +\infty]{} I_p.
$
Which finally gives us 
\begin{eqnarray}\label{lim:T-infty}
T_p(-\lambda) \xrightarrow[\lambda \rightarrow +\infty]{} 0_p.
\end{eqnarray}
Moreover, $\lambda\mapsto \frac{1}{w/\lambda+1}$ is differentiable for $w\in \mathbb R^+$, its derivative $w\mapsto \frac{w}{(w+\lambda)^2}$ is measurable  for $\lambda > 0$ and $|\frac{w}{(w+\lambda)^2}| \leq 1$, thus $T_p(-\lambda)$ is differentiable and 
$
    T_p'(-\lambda) = \int_{\mathbb R^+}\frac{\mu (dw)}{(w+\lambda)^2}.
$
We can prove the following limits the same way we proved \cref{lim:T-infty}
\begin{eqnarray}\label{lim:T'-infty}
    \lambda T_p'(-\lambda) \xrightarrow[\lambda \rightarrow +\infty]{} 0_p \quad \mathrm{and} \quad \lambda^2 T_p'(-\lambda) \xrightarrow[\lambda \rightarrow +\infty]{} I_p.
\end{eqnarray}
Hence by combining \cref{lim:T-infty,lim:T'-infty}, we get the limit of the predictive risk for large $\lambda$, 

$
    \lim\limits_{\lambda \rightarrow + \infty}r_\lambda^{test} (X_n) = \frac{\alpha^2}{p}\mathrm{Tr}[\widetilde{S}_p]+\sigma^2.
$
Let's now compute the limit of the predictive risk for small $\lambda$. Considering the definition of $T_p(0^-)$ and $T_p'(0^-)$ given in \cref{lem:limzero}, we directly get the limit of $r_\lambda^{test} (X_n)$ when $p < n$, 
$
    \lim_{\lambda \rightarrow 0, \; \lambda > 0} r_\lambda^{test} (X_n) = \sigma^2 + \frac{\sigma^2}{n}\mathrm{Tr}[\widetilde{S}_p T_p(0^-)].
$
Let's now suppose that $p > n$. We have from \cref{eq:T} that 
\begin{eqnarray*}
    T_{p}^{(j)}(-\lambda) = \frac{\kappa_j(\lambda)}{\lambda(\frac{1}{n}\mathrm{Tr}[\widetilde{D}_j] + \kappa_j(\lambda))}, \quad \mathrm{for } \quad 1\leq j \leq p,
\end{eqnarray*}
with $\kappa_j(\lambda) = \frac{\mathrm{Tr}[\widetilde{D}_j]}{\mathrm{Tr}[\widetilde{D}_j\widetilde{T}_n(-\lambda)]}$. Let's denote $\kappa(\lambda) = \underset{1\leq j \leq p}{\diag (\kappa_j(\lambda))}$, we then have a new expression of $T_p(-\lambda)$
\begin{eqnarray}\label{eq:T-kappa}
    T_p(-\lambda) = \frac{1}{\lambda}\kappa(\lambda)(\Sigma_n + \kappa(\lambda))^{-1}.
\end{eqnarray}
This equation allows us to derive a new formula for the predictive risk
\begin{eqnarray*}
r_\lambda^{test}(X_n) = \sigma^2 +  \frac{\alpha^2}{p}\Tr[\widetilde S_p \kappa (\lambda)(\Sigma_n + \kappa (\lambda))^{-1}] +  \left( \frac{\sigma^2}{n}- \frac{\lambda\alpha^2}{p}  \right) Tr[\kappa '(\lambda)\widetilde S_p\Sigma_n (\Sigma_n + \kappa (\lambda))^{-2}].
\end{eqnarray*}
Moreover, according to \cref{lem:limzero} $\kappa(\lambda)$ and $\kappa'(\lambda)$ admit limits when $\lambda$ tends to $0$. Thus we finish this proof by getting considering the definitions of $\kappa(0^+)$ and $\kappa'(0^+)$ from \cref{lem:limzero}.
\begin{eqnarray*}
    \lim_{\lambda \rightarrow 0, \; \lambda > 0} r_\lambda^{test}(X_n) = \frac{\alpha^2}{p}\mathrm{Tr}[\widetilde{S}_p \kappa (0^+) (\Sigma_n + \kappa (0^+))^{-1}] + \frac{\sigma^2}{n}\mathrm{Tr}[\kappa '(0^+)\widetilde{S}_p \Sigma_n (\Sigma_n + \kappa (0^+))^{-2}] + \sigma^2.
\end{eqnarray*}
\end{proof}

\begin{proof}[Proof of Corollary~\ref{cor:lambda_opt}]

Consider first the function
$$
\hat g(\lambda) = \hat r_\lambda^{test} (X_n) = \sigma^2 + \frac{\sigma^2}{n}\mathrm{Tr}\big[\widetilde{S}_p \Delta[Q_p(-\lambda)]\big] + \lambda\left( \frac{\lambda\alpha^2}{p} -  \frac{\sigma^2}{n} \right) \mathrm{Tr}\big[\widetilde{S}_p  \Delta[Q_p'(-\lambda)]\big].
$$ Since $Q_p(z) = (\widehat \Sigma_n - zI_p)^{-1}$ is the resolvent of the matrix $\widehat \Sigma_n$, one has that $Q'_p(z) = Q_p(z)^2$ and $Q''_p(z) = 2 Q_p(z)^3$. Consequently, the function 
$\hat g$ is differentiable for $\lambda > 0$, and 
we obtain that
$$
    \hat g'(\lambda) = 2 \Big( \frac{\lambda \alpha^2}{p}-\frac{\sigma^2}{n}\Big) \mathrm{Tr}\big[\widetilde{S}_p \Delta[ \widehat\Sigma_n Q_p(\lambda)^3]\big]
$$
using standard differential calculus.
Since the symmetric and positive semi-definite matrices $\widehat\Sigma_n$ and $Q_p(\lambda)^3$ are commuting, it follows that $\widetilde{S}_p \Delta[ \widehat\Sigma_n Q_p(\lambda)^3]$ is also a symmetric and positive semi-definite matrix. Thus $\mathrm{Tr}\big[\widetilde{S}_p \Delta[ \widehat\Sigma_n Q_p(\lambda)^3]\big] \geq 0$ which means that the sign of $\hat g'(\lambda)$ only depends on $\Big( \frac{\lambda \alpha^2}{p}-\frac{\sigma^2}{n}\Big)$.  Indeed, if $\lambda \leq \lambda_\ast$ then $\hat g'(\lambda) \leq 0$, and if $\lambda \geq \lambda_\ast$ then $\hat g'(\lambda) \geq 0$. This proves that $\lambda \mapsto \hat r_\lambda^{test} (X_n)$ attains its minimum at $\lambda = \lambda_\ast = \frac{\sigma^2p}{\alpha^2n}$.

Now, let us consider the function
$
g(\lambda) = r^{\mathrm{test}}_\lambda(X_n).
$
This function is differentiable for all $\lambda > 0$ and can be written as
$$
g(\lambda)
= \sigma^2 + \frac{\sigma^2}{n}\Tr\!\big[\widetilde S_p T_p(-\lambda)\big]
+ \lambda\Big( \frac{\lambda\alpha^2}{p} - \frac{\sigma^2}{n} \Big)
\Tr\!\big[\widetilde S_p T_p'(-\lambda)\big].
$$
By \cite{HLJ07}[Theorem~2.4], there exists a positive matrix-valued measure
$
\mu = (\mu_{ij})_{1 \le i,j \le p}
\quad \text{on } \mathbb R^+,
$
satisfying $\mu(\mathbb R^+) = I_p$, such that
$$
T_p(z) = \int_0^{+\infty} \frac{\mu(dw)}{w  - z}, \mbox{ for } z\in\mathbb C \setminus \mathbb R^+.
$$
Since the mapping $ z \mapsto (w-z)^{-1}$ is differentiable for all $w \in \mathbb R^+$ and its derivatives are uniformly bounded for $ z\in\mathbb C \setminus \mathbb R^+$, differentiation under the integral sign is justified. Consequently,
$$
T_p'(z) = \int_0^{+\infty} \frac{\mu(dw)}{(w-z)^2},
\qquad
T_p''(z) = \int_0^{+\infty} \frac{2\,\mu(dw)}{(w-z)^3}, \mbox{ for } z\in\mathbb C \setminus \mathbb R^+.
$$
Differentiating $g$ yields
\begin{align*}
g'(\lambda)
&= -\frac{\sigma^2}{n}\Tr\!\big[\widetilde S_p T_p'(-\lambda)\big]
+  \Big(2\frac{\lambda\alpha^2}{p} -\frac{\sigma^2}{n}\Big)\Tr\!\big[\widetilde S_p T_p'(-\lambda)\big]
- \lambda\Big( \frac{\lambda\alpha^2}{p} - \frac{\sigma^2}{n} \Big)
\Tr\!\big[\widetilde S_p T_p''(-\lambda)\big] \\
&= 2\Big( \frac{\lambda\alpha^2}{p} - \frac{\sigma^2}{n} \Big)
\Tr\!\Bigg[
\widetilde S_p
\int_0^{+\infty}
\Big( \frac{1}{(w+\lambda)^2} - \frac{\lambda}{(w+\lambda)^3} \Big)
\,\mu(dw)
\Bigg] \\
&= 2\Big( \frac{\lambda\alpha^2}{p} - \frac{\sigma^2}{n} \Big)
\Tr\!\Bigg[
\widetilde S_p
\int_0^{+\infty} \frac{w}{(w+\lambda)^3}\,\mu(dw)
\Bigg].
\end{align*}
Since $\mu(\mathbb R^+) = I_p$, the matrix
$
\int_0^{+\infty} \frac{w}{(w+\lambda)^3}\,\mu(dw)
$
is diagonal with $i$-th diagonal entry
$$
\int_0^{+\infty} \frac{w}{(w+\lambda)^3}\,\mu_{ii}(dw).
$$
As $\mu_{ii}$ is a positive measure and $\frac{w}{(w+\lambda)^3} > 0$ for all $w>0$, each diagonal entry is strictly positive. Hence, the matrix is positive semi-definite.
Since $\widetilde S_p$ is also positive semi-definite, we conclude that
$
\Tr\!\Bigg[
\widetilde S_p
\int_0^{+\infty} \frac{w}{(w+\lambda)^3}\,\mu(dw)
\Bigg] \geq 0.
$
Therefore, the sign of $g'(\lambda)$ depends solely on the factor
$
\frac{\lambda\alpha^2}{p} - \frac{\sigma^2}{n},
$
 and we conclude as before for $\hat g$, that the function $g(\lambda) = r^{\mathrm{test}}_\lambda(X_n)$ has a unique minimizer at $\lambda_\ast$.
This concludes the proof of Corollary~\ref{cor:lambda_opt}.
\end{proof}

\subsection{Random matrices with a variance profile in free probability}

We conclude this section by relating the computation of $T_p(z)$ to the notion of $\mathcal{R}$-transform in operator-valued free probability. The study of the spectral distribution of a generalized Wigner matrix (that is with a variance profile) dates back to  \cite{SHL96}, where, using Voiculescu's notion of asymptotic freeness \cite{VDN92}, Shlyakhtenko proved that independent generalized Wigner matrices  are asymptotically free with amalgamation over the diagonal. This property  has then further been studied in \cite{ACDGM17,male2020traffic} for independent permutation invariant matrices with variance profiles.

More formally, let $\mathbb{D}_p(\CC)^+$ (resp.\ $\mathbb{D}_p(\CC)^-$) denotes the set of diagonal matrices $\mathcal{Z}$ of size $p$ with diagonal complex entries having positive (resp.\ negative) imaginary parts. Recall that, for any square matrix $A$, we denote by $\Delta[A]$ the diagonal matrix whose diagonal entries are those of $A$. The operator-valued Stieltjes transform $\mathcal{G}_{A}$ of a Hermitian matrix $A$ is then defined as the map 
\begin{equation} \label{eq:opStieltjes}
	\begin{array}{cccc} 
		\mathcal{G}_{A}: & \mathbb{D}_p(\CC)^- & \to & \mathbb{D}_p(\CC)^+\\
		& \mathcal{Z} & \mapsto & \Delta\big[ (A - \mathcal{Z})^{-1} \big],\end{array} 
\end{equation}
and it is sometimes defined as $-\mathcal G_A$. 
The operator-valued $\mathcal R$-transform $\mathcal R_A$ of a Hermitian matrix $A$ is the unique analytic map satisfying, 
\begin{equation}\label{DefRTransform}
\mathcal{G}_{A}( \mathcal{Z}) = \Big(  -\mathcal{Z} + \mathcal{R}_{A}\big(- \mathcal{G}_{A}( \mathcal{Z}) \big) \Big)^{-1},
\end{equation}
for all $\mathcal Z$ in $\mathbb{D}_p(\CC)^-$ whose diagonal entries have imaginary parts large enough in absolute value. An efficient way to produce a deterministic equivalent for a random matrix $A$ is then to find a simple approximation $\mathcal{R}_{A}^\square$ of $\mathcal{R}_{A}$. In this manner, one can then approximate the operator-valued Stieltjes transform of $A$ by the solution of the fixed point equation \cref{DefRTransform} where $\mathcal{R}_{A}$ is replaced by $\mathcal{R}_{A}^\square$. This method also allows us to compute the spectrum of perturbations of $A$ by independent matrices, see for instance \cite{bigotmale}.

For a generalized Wigner matrix $W_p$ of size $p \times p$ with a symmetric variance profile $\Gamma_p$, Shlyakhtenko proves \cite{SHL96} that a good approximation of $\mathcal{R}_{W_p}$ is the deterministic linear map 
$
\mathcal{R}_{W_p}^\square ( \mathcal{Z}) = \mathrm{deg}\left( \frac{1}{p}\Gamma_p \mathcal{Z} \right),  \quad \mbox{for all} \quad  \mathcal{Z} \in \mathbb{D}_p(\CC)^-,
$
where  for a matrix $A$, we denote by $\mathrm{deg}(A)$ the diagonal matrix, whose $k$-diagonal element is the sum of the entries of the $k$-row of $A$.

To the best of our knowledge, apart from generalized Wigner matrices, there does not exist any other class of random matrices for which a simple approximation of the diagonal-valued $\mathcal R$-transform is known yet. Nevertheless, we ca now remark that, for  $\mathcal{Z} = z  I_p$ with $z \in \CC \setminus \RR^+$, the diagonal matrix $T_{p}(z)$ solution of the Dyson equation \cref{eq:Dyson} can be written as satisfying the fixed-point equation
$
T_{p}(z) = \left(- z I_p + \mathcal{R}_{\Sigma_n}^\square(-T_{p}(z))  \right)^{-1},
$
where $\mathcal{R}_{\Sigma_n}^\square$ is the non-linear map
$
\mathcal{R}_{\Sigma_n}^\square(\mathcal{Z}) = \mathrm{deg}\left(  \frac{1}{n}\Gamma_n^\top \left[I_n - \mathrm{deg}\left( \frac{1}{n}\Gamma_n \mathcal{Z} \right) \right]^{-1}\right),  \quad \mbox{for all} \quad  \mathcal{Z} \in \mathbb{D}_p(\CC)^-.
$
This suggests that $\mathcal{R}_{\Sigma_n}^\square$ is indeed a simple approximation of the operator-valued $\mathcal R$-transform of $\Sigma_n$.

\section*{Acknowledgments}
This work was supported by ANR-STARS funded project.

\bibliographystyle{alpha}
\bibliography{RidgeRegression_VarianceProfile-Final}

\end{document}